\documentclass{amsart}
\usepackage{amsmath}
\usepackage{amssymb}
\usepackage{amsthm}
\usepackage{amsfonts}
\usepackage{url}
\usepackage{hyperref}
\usepackage{appendix}
\usepackage{verbatim}
\usepackage{graphicx,color}
\usepackage{lineno}

\usepackage[curve,tips]{xypic}
\SelectTips{eu}{11}
\UseTips

\usepackage{tikz}
\usepackage{pst-knot}

\title{Absolute profinite rigidity, direct products, and finite presentability}
\author{M. R. Bridson}
\author{A. W. Reid}
\author{R. Spitler}
\address{\newline Mathematical Institute, 
\newline Andrew Wiles Building,
\newline University of Oxford,
\newline Oxford OX2 6GG, UK}
\email{bridson@maths.ox.ac.uk}
\address{\newline Department of Mathematics,
\newline Rice University, 
\newline Houston, TX 77005, USA}
\email{alan.reid@rice.edu,ryan.spitler@rice.edu}
 
\thanks{The first author was supported by the Clay Mathematics Institute, the second author by N.S.F. grant DMS-1812397 and the third author by the N.S.F. Postdoctoral Fellowship, DMS-2103335. For the purpose of open access, the authors have applied a CC BY public copyright licence to any author accepted manuscript arising from this submission.}

\def\-{\overline}

\def\wh{\widehat}

\def\G{\Gamma}

\def\im{{\rm{im}}}
\def\f{{\rm{f}}}

 \def\H{\mathbb{H}}
 \def\Z{\mathbb{Z}}
 \def\R{\mathbb{R}}
 \def\Q{\mathbb{Q}} 
  
 \def\C{\mathbb{C}}

\def\La{\Lambda}

 \def\SFS{{Seifert fibred space }}
 
 \def\SFSs{{Seifert fibred spaces }}

\DeclareMathOperator{\SL}{SL} \DeclareMathOperator{\PSL}{PSL}

\def\qed{ $\sqcup\!\!\!\!\sqcap$}

\def\tr{\mbox{\rm{tr}}\, }
\def\P{\mbox{\rm{P}}}

\def\G{\Gamma}

\def\<{\langle}
\def\>{\rangle}

\def\onto{\twoheadrightarrow} 

\newtheorem{theorem}{Theorem}[section]
\newtheorem{lemma}[theorem]{Lemma}
\newtheorem{corollary}[theorem]{Corollary}

\newtheorem{proposition}[theorem]{Proposition}

\theoremstyle{definition} 
\newtheorem{definition}[theorem]{Definition}
\newtheorem{remark}[theorem]{Remark}
\newtheorem{example}[theorem]{Example}

\numberwithin{equation}{section}

\def\L{\Lambda}
\def\D{\Delta}

\begin{document}
 

\begin{abstract}  
We prove that there exist finitely presented,
residually finite groups that are profinitely rigid in the class of all finitely presented
groups but not in the class of all finitely generated groups. These groups are of the form $\G\times\G$
where $\G$ is a profinitely rigid 3-manifold group; we describe a family of such groups with the property
that if $P$ is a finitely generated, residually finite group with $\wh{P}\cong\wh{\G\times\G}$ then there is
an embedding $P\hookrightarrow\G\times\G$ that induces the profinite isomorphism; 
in each case there are infinitely many non-isomorphic possibilities for $P$.  
\end{abstract}




\maketitle

%
%
%
%

\section{Introduction} 
\label{intro}
The quest to understand the extent to which finitely generated groups are determined by their finite images has
been greatly invigorated in recent years with input from low-dimensional geometry and topology.
In our papers \cite{BMRS1} and \cite{BMRS2} with D.~B. McReynolds, we provided the first examples of finitely generated, residually finite, full-sized groups $\Gamma$ that are  
{\em profinitely rigid}:  
 for finitely generated, residually finite groups  $\Lambda$, if $\wh{\Lambda}\cong \wh{\Gamma}$ then $\Lambda \cong \Gamma$, where $\wh{\L}$ denotes the profinite completion
of $\L$.  
The following theorem provides the first examples of finitely presented groups that are profinitely rigid among
finitely presented groups but not among finitely generated groups.   Like the examples in 
 \cite{BMRS1, BMRS2, BRprasad, TCW}, the groups $\G$ in this theorem are 3-manifold groups with particular arithmetic properties. 
Here,
 $S^2(p,q,r)$ denotes the quotient $\mathbb{H}^2/\Delta(p,q,r)$ of the hyperbolic plane by the 
 triangle group $\D(p,q,r)$  
 (the index-2 orientation-preserving subgroup of the reflection group associated to a hyperbolic triangle with interior angles $\pi/p$, $\pi/q$ and $\pi/r$). 
 
\begin{theorem}
\label{t:main} There exist finitely presented, residually finite groups $\G$ with the following properties:
\begin{enumerate}
\item $\G\times\G$ is profinitely rigid among all finitely presented, residually finite groups. 
\item There exist infinitely many non-isomorphic finitely generated groups $\Lambda$ such that  $\wh{\Lambda}\cong \wh{\G}\times\wh{\G}$.
\item If $\Lambda$ is as in (2), then there is an embedding
$\Lambda\hookrightarrow\G\times\G$ that induces the isomorphism $\wh{\Lambda}\cong \wh{\G\times\G}$
(in other words,  $\Lambda\hookrightarrow\G\times\G$ is a Grothendieck pair). 
\end{enumerate}  
If $M$ is any Seifert fibred space 
with base  orbifold $S^2(3,3,4)$ or $S^2(3,3,6)$ or $S^2(2,5,5)$,  then $\G=\pi_1M$ has these properties. 
\end{theorem}

In general, for a fixed group $\G$, there can be uncountably many finitely generated, residually finite groups
$H$ with $\wh{H}\cong\wh{\G}$; see \cite{Nek} and \cite{pyber} for example. In other settings, for example nilpotent
groups \cite{Pick}, there are only finitely many such $H$, up to isomorphism.
The  groups $\G\times\G$ in Theorem \ref{t:main} provide the first examples of groups where the
number of such $H$ is countably infinite; this property is assured by parts (2) and (3) of Theorem \ref{t:main}.
We express this in the language of \cite{GZ}, where the
{\em profinite genus} of a finitely generated,  residually finite group $\G$ 
is defined to be the set of isomorphism classes
of finitely generated, residually finite groups $H$ such that $\wh{H}\cong\wh{\G}$.

\begin{corollary}
There exist finitely presented, residually finite groups whose profinite genus  
is countably infinite.
\end{corollary}

A motivating question in the study of profinite rigidity is to determine whether  the
free group $F_r$ is profinitely rigid when $r\ge 2$. It is known that $F_r\times F_r$ is not
profinitely rigid, indeed there are Grothendieck pairs $P\hookrightarrow F_r\times F_r$
with $P\not\cong F_r\times F_r$ finitely generated. 
There are no such pairs with $P$ finitely presented,
and it seems reasonable to conjecture that $F_r$ satisfies Theorem \ref{t:main}. In order to prove this,
one would first have to show that $F_r$ was profinitely rigid, but even this would not be enough.

The groups $\G$ for which we shall prove Theorem \ref{t:main} are  fundamental groups of \SFSs
whose base orbifold is $S^2(p,q,r) = \H^2/\D(p,q,r)$, where $\D(p,q,r)$ is one of
\begin{equation}\label{list-top}
\Delta(3,3,4), \ \Delta(3,3,5),\ \Delta(3,3,6),\ \Delta(2,5,5), \ \Delta(4,4,4).     
\end{equation}

Theorem \ref{t:main} summarizes the main contribution of this paper, but we shall prove a
number of other  results  that are of independent interest. First, we prove the following theorem, which 
provides the first infinite family of closed 3-manifolds whose fundamental groups are full sized (i.e. they contain a free subgroup of rank $2$) 
and profinitely rigid.  
For this, we augment the above list with
\begin{equation}\label{list-bottom}
\Delta(2,3,8), \Delta(2,3,10),\ \Delta(2,3,12),\ \Delta(2,4,5), \ \Delta(2,4,8).
\end{equation}

\begin{theorem}\label{SFS_rigid} 
Let $\Delta(p,q,r)$  be a triangle group from  list (\ref{list-top}) or (\ref{list-bottom}), and let $M$ be a Seifert fibred space with base $S^2(p,q,r)$.  
Then $\pi_1 M$ is profinitely rigid.
\end{theorem}

The following pair of contrasting theorems explain how finite presentability emerges as a 
key determinant of profinite rigidity. The first of
these theorems is the key to promoting abstract isomorphism of profinite completions to isomorphisms induced
by inclusions (i.e. Grothendieck pairs). This control is vital in the proofs of parts (1) and (2) of Theorem \ref{t:main}.

\begin{theorem} \label{t3}
Let $\Delta(p,q,r)$ be a triangle group from list  (\ref{list-top}) or  (\ref{list-bottom}), let $M$ be a 
Seifert fibred space with  base orbifold $S^2(p,q,r)$ and let $\G=\pi_1M$.
Then, for  every finitely generated, residually
 finite group $\L$ with $\wh{\L}\cong\wh{\G\times\G}$, there is
 an embedding $\L\hookrightarrow\G\times\G$ that induces the isomorphism.
 \end{theorem}

Recall that a subgroup of $\PSL(2,\C)$ is termed {\em elementary} if its limit set in $\mathbb{CP}^1$ consists of at most two points. Otherwise, the group is called {\em non-elementary}. A discrete subgroup is elementary if and only if it is 
virtually abelian, which accords with the terminology used for Gromov-hyperbolic groups. 
Throughout, by a Fuchsian group 
we will always mean a finitely generated, discrete, non-elementary subgroup of $\PSL(2,\R)$.

\begin{theorem}\label{t:not-fp}
For every Fuchsian group $F$
and every Seifert fibred space  $M$ with base orbifold $\H^2/F$, there are no Grothendieck pairs $\L\hookrightarrow
\pi_1M\times\pi_1M$ with $\L\neq \pi_1M\times\pi_1M$ finitely presented.
\end{theorem}

To complete the proof of Theorem \ref{t:main} we prove the following result.

\begin{theorem}\label{t:not-GR}
If $\Delta(p,q,r)$ is a triangle group from  list (\ref{list-top}), 
then there are infinitely many
Seifert fibred spaces with  base orbifold $S^2(p,q,r)$ whose fundamental group $\G$ has the 
property that there are infinitely many non-isomorphic, finitely generated groups $\La$
and inclusions $\La\hookrightarrow \G\times \G$ inducing isomorphisms $\wh{\La}\cong\wh{\G\times \G}$. 
 \end{theorem}

For the moment, $S^2(3,3,4),\ S^2(3,3,6)$ and $S^2(2,5,5)$ 
are the only orbifolds for which we can prove that the conclusion of Theorem \ref{t:main}
holds for the fundamental group of {\em every} \SFS with the given base. 
The obstruction to extending this statement to the other bases  from lists 
(\ref{list-top}) and (\ref{list-bottom}) is discussed in Section \ref{s:last}.

Let us turn to some of the key ideas behind these results. In \cite{BMRS2} we proved that the 
triangle groups  $\D$ in lists ({\ref{list-top}) and ({\ref{list-bottom}) are profinitely rigid. We did so  by employing the template from \cite{BMRS1}, starting 
 with a proof that the groups  in ({\ref{list-top}) are Galois rigid (this notion is recalled in Section \ref{galois}). 
Theorem \ref{SFS_rigid} follows a similar path, and we begin by establishing Galois rigidity for the  groups $\pi_1M$ in Theorem \ref{SFS_rigid}.
It is not the case that  an arbitrary
central extension of $\Delta$ will be profinitely rigid \cite{piwek}, but we shall prove that it
does follow when the centre is cyclic.  To do this, we appeal to the work of
Wilkes  \cite{Wil} (recalled in Theorem \ref{main_wilkes} below), who proved that the 
 fundamental groups of most Seifert fibred spaces can be distinguished from 
 the fundamental groups of other compact $3$-manifolds by their profinite completions.  

The key step of reducing abstract profinite isomorphism to the study of Grothendieck pairs (Theorem \ref{t3}) is explained in Section \ref{s:t3}. Roughly speaking, we extend the ideas used in \cite{BMRS1} to cover direct products 
of Galois rigid groups, promoting profinite isomorphisms $\wh{P}\cong \wh{\D\times\D}$ to Grothendieck
pairs ${P}\hookrightarrow{\D\times\D}$ under suitable arithmetic conditions. 
We then consider the effect of  taking central extensions.

For Theorem \ref{t:not-GR}}, not all of the groups from Theorem \ref{SFS_rigid} can be used:
we need constraints on the Seifert invariants that ensure  
$\G=\pi_1M$ has finite index in $[G,G]$ where $G$ is a group that maps onto a non-elementary hyperbolic group and has $H_1(G,\Z)$ finite and $H_2(G,\Z)=0$.

These conditions allow us to apply a homological argument, adapted from an idea of
Bass and Lubotzky \cite{BL}, that constructs an epimorphism $\pi_1M\to Q$ where $Q$ is 
a finitely presented infinite group with $\wh{Q}=1$ and $H_2(Q,\Z)=0$;
 this is explained in Section \ref{s:homol-argument}. 
With this epimorphism in hand, an argument originating in the work of Platonov and Tavgen \cite{PT} and developed
in \cite{BG} and \cite{BL},  allows us to construct a finitely generated Grothendieck
pair $\Lambda\hookrightarrow 
\G\times\G$ by taking a fibre product. 
In fact, we are able to construct infinitely many non-isomorphic $Q$ and from these we obtain infinitely many possibilities 
for $P$, with $\G$ fixed. In the language of \cite{B-jems}, this shows that $\G\times \G$ has infinite strong profinite genus among finitely generated groups. The fact that there are no finitely presented Grothendieck pairs $\Lambda\hookrightarrow\G\times \G$ 
ultimately boils down to the fact every finitely  presented subgroup of a direct product of Fuchsian
groups is closed in the profinite topology \cite{BW}; see Section \ref{products}. We shall prove
in Theorem \ref{p:fills2} that the statement  of Theorem \ref{t:not-fp} remains valid for products of any finite number of
copies of $\pi_1M$. \\[\baselineskip]
\noindent{\bf Acknowledgements:}~{\em Thanks are due to the referees for their many helpful comments and suggestions. We are also
grateful to multiple institutions and hosts for their hospitality and support as this work developed: the Mathematisches Forschungsinstitut Oberwolfach;
the Max-Planck-Institut f\"ur Mathematik, Bonn; the University of Auckland; the University of Oxford; the organizers of the conference
``Groups of Dynamical Origins, Automata and Spectra" in Les Diablerets; 
the Instituto de Ciencias Matem\'aticas, Madrid, particularly the AgolLab; and finally, the organizers of the July 2023 conference ``Group actions and low-dimensional topology" in El Barco de \'Avila, and the townspeople of El Barco de \'Avila.
The third author would also like to thank Neil Hoffman for asking  whether Theorem \ref{SFS_rigid} might hold.}

\section{Preliminaries}
\label{prelims}

Throughout this section, $\Gamma$ will be a finitely generated, residually finite group and $\widehat{\Gamma}$ its profinite completion. We will denote by $Z(\Gamma)$ and $Z(\widehat{\Gamma})$ the centres of $\G$ and $\wh{\G}$ respectively.
Throughout the paper, if $G$ is a profinite group   and $K\subset G$ a subset, then $\overline{K}$ will denote the closure of $K$ in $G$.

\subsection{Centres}
\label{centre}
The relationship between the centre of $\Gamma$ and the centre of $\wh{\Gamma}$ can be somewhat complicated, but the following observation is straightforward.

\begin{lemma}
\label{centres}
$Z(\Gamma) = Z(\widehat{\Gamma}) \cap \Gamma$.
\end{lemma}

\begin{proof}
As $\Gamma < \widehat{\Gamma}$ is a subgroup, that $Z(\widehat{\Gamma}) \cap \Gamma \subset Z(\Gamma)$ is clear.
On the other hand, since $\Gamma < \widehat{\Gamma}$ is dense, any element of $\widehat{\Gamma}$ fixed by the conjugation action of $\Gamma$ will be fixed by the conjugation action of all of $\widehat{\Gamma}$. Thus $Z(\Gamma) \subset Z(\widehat{\Gamma})$.
\end{proof}

A similar argument shows that $\overline{Z(\Gamma)} \subset Z(\widehat{\Gamma})$, but in general this inclusion may be strict. The centre of the profinite completion can become unexpectedly larger as the following example shows.

\begin{example}
If $\G = \SL(3,\mathbb{Z})$, then $Z(\Gamma) = \{1 \}$ is trivial. On the other hand, there are infinitely many $p$ for which $\mathbb{Z}_p$ (the $p$-adic integers) contains a primitive cube root of unity, in which case the centre of $\SL(3,\mathbb{Z}_p)$ is a cyclic group $C_3$ of order $3$. Therefore $Z(\widehat{\Gamma}) = Z(\prod_p \SL(3,\mathbb{Z}_p))$ contains a countable direct product of copies of $C_3$.
\end{example}

We will  need   the following result on the centres of profinite completions of Fuchsian groups. 
This seems to be well-known, but we could not find a reference and therefore include a proof.

\begin{theorem} \label{prof-centre-fuchs}
Let $\Gamma$ be a finitely generated Fuchsian group. Then $Z(\widehat{\Gamma}) = \{1\}$.
\end{theorem}

\begin{proof} We first assume that $\Gamma$ is torsion-free, so $\G$ is either a free group or surface group. In the case when $\Gamma$ is free, that $\widehat{\Gamma}$ is centre-free follows directly from \cite[Theorem 3.16]{ZM}. When
$\Gamma$ is a surface group, we can decompose $\Gamma = A *_C B$ where $A$ and $B$ are free groups and $C$ is infinite cyclic. By \cite{Sc2} $\Gamma$ is LERF, and so the profinite topology on $\G$ induces the full profinite topology on $A$, $B$ and $C$ and, furthermore,
$\wh{\G}$ decomposes as the profinite free product with amalgamation $\wh{A} \sqcup_{\wh{C}} \wh{B}$ (see \cite[Chapter 9.2]{RZ}).
Using the previously established centre-freeness of the profinite completion of free groups, we can apply \cite[Theorem 3.16]{ZM} once again to deduce that $\wh{\G}$ is also centre-free.  

We now assume that $\G$ contains elements of finite order. 
It is well-known that $\Gamma$ admits surjections onto finite simple groups of the form $\PSL(2,q)$ (for example by taking a particular realization of $\Gamma$ with algebraic traces and reducing modulo primes of the trace-field) whose kernel is torsion-free (and hence a free group or a surface group). Since $\PSL(2,q)$ is centre-free, $\wh{\G}$ is an extension of a centre-free group by a centre-free group so is itself centre-free.
\end{proof}

\subsection{Goodness}\label{good}

A group $\Gamma$ is {\em good} in the sense of Serre \cite{Ser} if for any finite $\Gamma$-module $M$, the natural map $\Gamma\rightarrow \wh{\G}$ induces isomorphisms:

\[
\mathrm{H}^q(\widehat{\Gamma},M) \xrightarrow{\sim} \mathrm{H}^q(\Gamma,M)
\]
from the continuous cohomology of $\widehat{\Gamma}$ to the group cohomology of $\Gamma$, both with coefficients in $M$.

We will need the following examples of good groups (see \cite{GJZ} for the case of Fuchsian groups and \cite[Theorem 7.3]{Re} for the case of fundamental groups of compact $3$-manifolds).

\begin{theorem} \label{good-groups}\label{good1}
If $\Gamma$ is a finitely generated Fuchsian group or the fundamental group of a compact $3$-manifold, then $\Gamma$ is good.
\end{theorem}

A useful consequence of goodness that we will make use of is the following (see \cite[Corollary 7.6]{Re} for instance). 

\begin{theorem} \label{good-tors-free}\label{good2}
If $\Gamma$ is a finitely generated residually finite group of finite cohomological dimension  that is good, then $\widehat{\Gamma}$ is torsion-free.
\end{theorem} 

\subsection{Central Extensions} 
\label{central}

A further consequence of goodness is the following (see \cite[Corollary 6.2]{GJZ}): if $G$ is a residually finite, good group and $H\to G$
is an epimorphism with finitely generated, residually finite kernel, then $H$ is
residually finite.  We make use of this in the following way.

\begin{lemma} 
\label{full_profinite}
Suppose that $H$ is residually finite with finitely generated centre $Z$, and that
$H/Z$ is good. Then $H$ induces the full profinite
topology on $Z$. 
\end{lemma}

\begin{proof}
If $Z'<Z$ is a subgroup of finite index in $Z$,  
then $Z'$ is normal in $H$ and we have a central extension
$$
1\to Z/Z' \to H/Z' \to H/Z\to 1.
$$
By the remark above, since $H/Z$ is good, we deduce that $H/Z'$ is also residually finite. Hence there is a homomorphism from $H/Z'$ to a finite
group that is injective on $Z/Z'$. If $K$ is the kernel of the composition
of this map with $H\to H/Z'$, then $K\cap Z=Z'$.
\end{proof}

\begin{corollary}\label{c:zhat}
\label{exact}
$1\to \wh{Z}\to \wh{H}\to \wh{H/Z}\to 1$ is exact.
\end{corollary}

\begin{remark} \label{remark_closure}
The exactness in Corollary \ref{exact} holds more generally as can be seen from \cite[Exercise 2b]{Ser}:
if $1 \to K \to \Gamma \to Q \to 1$ is exact with  $Q$ and $K$ finitely generated and residually finite, and if
$Q$ is good, then $\wh{K}\cong\overline{K}$ in $\widehat{\Gamma}$.
\end{remark}

\begin{proposition}\label{p:sameZ}
Let $B$ be a good group such that $\wh{B}$ has trivial centre and let $E_1, E_2$
be central extensions of $B$ with finitely generated kernels $Z_1,Z_2$.
 Then $\wh{E_1}\cong
\wh{E_2}$ implies $Z_1\cong Z_2$.
\end{proposition}

\begin{proof} For each extension we obtain a short exact sequence as in Corollary \ref{c:zhat}. Because
$\wh{B}$ has trivial centre, $\wh{Z_i}$ is the centre of $\wh{E_i}$, so $\wh{E_1}\cong
\wh{E_2}$ implies $\wh{Z_1}\cong\wh{Z_2}$, hence $Z_1\cong Z_2$. 
\end{proof}

The hypothesis that $Z_1$ and $Z_2$ are finitely generated is unnecessary if we assume $B$ is finitely
presented and $E_1, E_2$ are finitely generated because of the following well known fact.

\begin{lemma}\label{l:Zfg}
Let $1\to Z\to B\to C\to 1$ be exact with $Z$ central. If $C$ is finitely presented and $B$ is finitely generated then $Z$ is
finitely generated.
\end{lemma}

\begin{proof}
Since $C$ is finitely presented and $B$ is finitely generated, $Z$ is finitely generated as a normal subgroup of $B$. But the action of $B$ by 
conjugation on $Z$ is trivial, so any finite
subset of $Z$ that generates it as a normal subgroup already generates it as a subgroup.\end{proof} 

\section{Galois Rigidity}\label{galois}

We begin by recalling what we need about Galois rigidity from \cite{BMRS1} and \cite{BMRS2}. As in these papers, 
let $\phi\colon \mathrm{SL}(2,\mathbb{C}) \to \mathrm{PSL}(2,\mathbb{C})$ be the quotient homomorphism, and if $H$ is a finitely generated subgroup of $\PSL(2,\mathbb{C})$,  set $H_1 = \phi^{-1}(H)$. It will be convenient to say that $H$ is 
{\em Zariski-dense} in $\PSL(2,\C)$ when what we actually mean is that $H_1$ is a Zariski-dense subgroup of $\SL(2,\C)$. The \textit{trace-field} of $H$ is defined to be the field 
\[ K_H=\mathbb{Q}(\mathrm{tr}(\gamma)~\colon~ \gamma \in H_1). \] 
If $K_H$ is a number field with ring of integers $R_{K_H}$, we say that $H$ has {\em integral traces} if $\tr(\gamma)\in R_{K_H}$ for all $\gamma \in H_1$. 

Suppose that $H$ is a finitely generated group and  $\rho\colon H\rightarrow \PSL(2,\C)$ a Zariski-dense representation with $K=K_{\rho(H)}$ a number field of degree $n_K$. If $K=\Q(\theta)$ for some algebraic number $\theta$, then the Galois conjugates of $\theta$, say $\theta=\theta_1,\dots,\theta_{n_K}$, provide embeddings $\sigma_i\colon K\to\C$ defined by $\theta\mapsto\theta_i$.  These in turn can be used to build $n_K$ Zariski-dense non-conjugate representations $\rho_{\sigma_i}\colon H \to \PSL(2,\C)$ with the property that $\tr(\rho_{\sigma_i}(\gamma))=\sigma_i(\tr\rho(\gamma))$ for all $\gamma\in H$. We refer to these as {\em Galois conjugate representations}. 
The existence of these Galois conjugates shows that $|\mathrm{X}_{\mathrm{zar}}(H,\mathbb{C})|\geq  n_{K_{\rho(H)}}$, where 
$\mathrm{X}_{\mathrm{zar}}(H,\C)$ denotes the set of Zariski-dense representations $H\to\PSL(2,\C)$ up to conjugacy.

\begin{definition}\label{def:galois-rigid}
Let $\G$ be a finitely generated group and $\rho\colon \G\to \PSL(2,\C)$ a Zariski-dense representation whose trace field $K_{\rho(\G)}$ is a number field. If
 $|\mathrm{X}_{\mathrm{zar}}(\G,\mathbb{C})|= n_{K_{\rho(\G)}}$, we say that $\G$ is {\em Galois rigid}  (with associated field $K_\G$).
\end{definition} 

Zariski-dense representations $\beta : \G \to \PSL(2,\C)$ are irreducible, so it follows from \cite{CS} (see also \cite[Section 3]{BZ}) that they are determined  up to conjugacy by their character. With this in mind, we shall sometimes abuse notation by
writing $\beta\in \mathrm{X}_{\mathrm{zar}}(\G,\mathbb{C})$ when what we mean is that $\beta$ is a Zariski-dense representation. Likewise, it is sometimes convenient to refer to elements of $\mathrm{X}_{\mathrm{zar}}(\G,\mathbb{C})$ as if they were representations.

In our papers \cite{BMRS1, BMRS2} with McReynolds, the Galois rigidity of certain Fuchsian and Kleinian groups played a crucial role in the proof of profinite rigidity.  We will need to extend some of \cite{BMRS1, BMRS2} to the settings of central extensions of certain Fuchsian groups and direct products of these Fuchsian groups.  The discussion below,  in particular the number theoretic set-up, is guided by \cite{BMRS2}.

 We fix a real quadratic number field $K\subset \R$ with $\sigma\colon K\to \R$ the non-trivial Galois embedding, a quaternion algebra $B/K$, and a maximal order $\mathcal{O} < B$. Since $K$ has two real places $v_1$ (the identity place) and $v_2$ (associated to $\sigma$), we can prescribe that $B$ be ramified at either of $v_1$ or $v_2$, and unramified at the other; denote these two possibilities by $B_1$ and $B_2$ respectively. If these $B_i$ are only additionally ramified at a finite place $\omega$ with residue field of characteristic $p$ and $\omega$ is the unique such place, then although $B_1$ and $B_2$ are not isomorphic (over $K$), there is an extension of $\sigma$ that maps $B_1$ to $B_2$. 
In this situation, and up to this ambiguity, we can identify $B$ with either of the $B_i$ (for $i=1,2$). 

Now suppose that $\rho : B\to M(2,\R)$ is a representation and $\mathcal{O}^1$ the elements of norm one in $\mathcal{O}^1$, so that $\rho(\mathcal{O}^1) < \SL(2,\R)$. Let $\Gamma$ be a finitely generated, residually finite group with a representation 
$f: \Gamma \to \P\rho(\mathcal{O}^1)<\PSL(2,\C)$ whose image is Zariski-dense with trace-field  $K$ (in fact this is automatic given the hypothesis that $B$ is ramified at a real place of $K$).  
With this preamble established, the following result can be readily extracted from \cite[Theorem 4.8 and Corollary 4.11]{BMRS1}.

\begin{theorem}
\label{galois_rigid_reps}
Let $f: \Gamma \to \P\rho(\mathcal{O}^1)$ be as above and assume that $\Gamma$ is Galois rigid. If $\Sigma$ is a finitely generated, residually finite group such that $\widehat{\Sigma} \cong \widehat{\Gamma}$, then:
\begin{itemize}
\item[(i)]
$\Sigma$ is Galois rigid with associated field $K$ and quaternion algebra $B$.
\item[(ii)] 
If $B$ has type number $1$, and if $\mathrm{Ram}(B) = \{v_2,\omega\}$ where $v_2$ is the real place described above, and $\omega$ is a finite place as above, then there is a homomorphism $f': \Sigma \to \P\rho(\mathcal{O}^1)\subset \PSL(2,\C)$ with Zariski-dense image.
\end{itemize}
\end{theorem}

We now investigate direct products and Galois rigidity. We continue to assume that $\G$ is a finitely generated, residually finite group that is Galois rigid, with a representation $f: \Gamma \to \P\rho(\mathcal{O}^1)<\PSL(2,\C)$ (with $\mathcal{O} \subset B$ as described above)  whose image is Zariski-dense. We will consider $\G \times \G$. Let $\G_1 = \G \times 1$ and $\G_2 = 1 \times \G$, and for $i =1,2$ let $\pi_i: \G \times \G \to \G_i$ denote the projection homomorphism. For each $i$, let $\phi_i = f \circ \pi_i$ and $\overline{\phi}_i = \overline{f} \circ \pi_i$ where $\overline{f}$ is the Galois conjugate of the representation  $f$. Note that although $\G$ is Galois rigid, $\G \times \G$ is not.

\begin{lemma}\label{l:thru-proj} 
 \begin{enumerate}
 \item For every representation of a direct product $\alpha : H_1\times H_2\to {\rm{PSL}}(2,\C)$, if $\im (\alpha)$
 is Zariski-dense, then $\alpha(H_i)=1$ for exactly one of $i=1$ or $i=2$.
 \item If a group $H$ has non-trivial centre $Z$, then $\alpha(Z)=1$ for
 every representation $\alpha : H\to {\rm{PSL}}(2,\C)$ with  $\im (\alpha)$
 being Zariski-dense.
 \end{enumerate} 
 \end{lemma}
 
 \begin{proof} If $\gamma\in \alpha(H_1)$ is non-trivial, then $\alpha(H_2)$ lies in the centraliser of $\gamma$
 in $ {\rm{PSL}}(2,\C)$,  which is abelian. Thus $A=\alpha(H_2)$ is a normal abelian subgroup of the Zariski-dense
 group $\alpha(H_1\times H_2)$. But Zariski-dense subgroups of ${\rm{PSL}}(2,\C)$ do not contain
 any non-trivial subgroups of this form. Briefly, since $A$ is a normal subgroup of $\alpha(H_1\times H_2)$, the Zariski closure of $A$ is a normal subgroup of the Zariski closure of $\alpha(H_1\times H_2)$, which is $\PSL(2,\C)$, hence it agrees with
 $\PSL(2,\C)$. But $A$ being abelian implies that the Zariski closure is abelian, and this is a contradiction, so
 (1) is proved. The proof of (2) is similar.
 \end{proof}

\begin{corollary}
\label{products_reps}
For each $\psi \in \mathrm{X}_{\mathrm{zar}}(\G \times \G,\mathbb{C})$ there is a unique $i$ so that $\psi$ is conjugate to exactly one of $\phi_i$ or $\overline{\phi}_i$. Thus $\G \times \G$ has exactly $2$ distinct Zariski-dense representations in $\PSL(2,\C)$ up to conjugation and taking Galois conjugates. In particular, $|\mathrm{X}_{\mathrm{zar}}(\G \times \G,\mathbb{C})| = 4$ and all of these representations have integral traces.
\end{corollary}

\begin{proof}  
Let $\psi \in \mathrm{X}_{\mathrm{zar}}(\G \times \G,\mathbb{C})$. Since $\im (\psi)$ is Zariski-dense, by Lemma \ref{l:thru-proj}, there is a unique $i$ so that $\psi$ factors through $\pi_i$. By hypothesis, every Zariski-dense representation of $\G_i$ is conjugate to exactly one of $f$ or $\overline{f}$. Hence $\psi$ is conjugate to exactly one of $\phi_i$ or $\overline{\phi}_i$. Because representations factoring through distinct $\G_i$ have distinct kernels, they cannot be conjugate. Also $\phi_i$ and $\overline{\phi}_i$ cannot be conjugate as they induce distinct characters, so we have $|\mathrm{X}_{\mathrm{zar}}(\G \times \G,\mathbb{C})| = 4$. Each  
$\phi_i$ and  $\overline{\phi}_i$ has integral traces because $f$ and $\overline{f}$ do.
\end{proof}

This corollary tells us that although $\G \times \G$ is not Galois rigid, it still has only finitely many Zariski-dense representations in $\PSL(2,\C)$ up to conjugation, and
we can account for them. Using ideas from \cite{BMRS1} and \cite{Spit} that extend Theorem \ref{galois_rigid_reps}, we can exploit this representation rigidity 
for $\G \times \G$ to prove the following theorem.

\begin{theorem}
\label{products_prof_thm}
With $f: \Gamma \to \P\rho(\mathcal{O}^1)$  
and $\G \times \G$ as above, and the continuing assumption that $\G$ is Galois rigid, let $\Sigma$ be a finitely generated, residually finite group such that $\widehat{\Sigma} \cong \widehat{\G \times \G}$. Then
\begin{itemize}
\item[(i)]
$\Sigma$ has exactly $2$ distinct Zariski-dense representations, $\psi_j: \Sigma\to \PSL(2,\C),\ j = 1,2$, up to conjugation and the taking of Galois conjugates.
\item[(ii)]
$|\mathrm{X}_{\mathrm{zar}}(\Sigma,\mathbb{C})| = 4$ and the image of each Zariski-dense
 representation has trace field $K$.
\item[(iii)] 
If $B$ has type number $1$ and if $\mathrm{Ram}(B) = \{v_2,\omega\}$ where $v_2$ is the real place described above, and if $\omega$ is a finite place as above, then
for $j=1,2$, the image of $\psi_j$ is contained in $\P\rho(\mathcal{O}^1)$.
\end{itemize}
\end{theorem}

\begin{proof}
For $i=1,2$ we have the representations $\phi_i$ and $\overline{\phi}_i$ of $\G \times \G$ in $\P B^1<\PSL(2,\C)$ and  the trace field of the image of each is $K$. From Corollary \ref{products_reps}, we know that these are the only Zariski-dense representations of $\G \times \G$ in $\PSL(2,\C)$ up to conjugation. Each of these representations has integral traces, so the image of $\G \times \G$ in each case will be $p$-adically bounded for any $p$-adic place of $K$.

With this information in hand, one can apply the  arguments from \cite[Theorem 4.8]{BMRS1} and \cite[Theorem 6.1]{Spit} to the fixed profinite isomorphism $\widehat{\Sigma} \cong \widehat{\G \times \G}$ to obtain similar collections for $\Sigma$. The precise statement is as follows:

\begin{proposition}
For $j=1, 2$ there is a number field $L_j$, a quaternion algebra $A_j$, and a Zariski-dense representation $\psi_j: \Sigma \to \P A_j^1 < \PSL(2,\C)$ such that $\psi_j(\Sigma)$ has trace field $L_j$. Furthermore,
\begin{itemize}
\item[(a)]
every Zariski-dense representation of $\Sigma$ is conjugate in $\PSL(2,\C)$ to a Galois conjugate of  $\psi_1$ or $\psi_2$;
\item[(b)]
there is a bijection between sets of finite places, $\tau: V_{L_1}^\f \sqcup V_{L_2}^\f \to V_K^\f \sqcup V_K^\f$ so that if $v$ is a $p$-adic place of $L_j$, then $\tau(v)$ is a $p$-adic place of $K$ with $L_{j,v}\cong K_{\tau(v)}$ and $A_j \otimes_{L_j} L_{j,v} \cong B \otimes_K K_{\tau(v)}$;
\item[(c)] 
each $\psi_j(\Sigma)$ has integral traces.
\end{itemize}
\end{proposition}

Continuing with the proof of Theorem \ref{products_prof_thm}, 
note that {\em a priori} the bijection $\tau$ in (b) does not  need to restrict to separate bijections of the $V_{L_j}^\f$ with $V_K^\f$, as $L_j$ may be paired with different copies of $V_K^\f$ as the place $v$ varies. Also note that it is not {\em a priori} necessary that the number fields $L_1$ and $L_2$ or the quaternions algebras $A_1$ and $A_2$ be isomorphic. However, with the strict arithmetic assumptions we have made on $K$ and $B$ we will shortly show that in fact $L_j \cong K$ and $A_j \cong B$ for each $j$.

Now we can immediately see that (i) of the theorem follows from (a) of the proposition. To prove (ii) it is sufficient to show that each $L_j \cong K$. Without loss of generality, we will show it for $L_1$. Note that if $p$ is a rational prime which splits in $K$, then for each $p$-adic place $w$ of $K$ we have $K_w \cong \Q_p$. Thus condition (b) above implies $L_1$ is completely split over $p$. Since this holds for every split prime $p$, and $K$ is Galois over $\Q$, by \cite[p.~108, Cor to Thm 31]{Mar} we know that $L_1$ is isomorphic to a subfield of the quadratic field $K$. Next, if $p$ is a prime that is inert in $K$, then for each $p$-adic place $w$ of $K$ we have $[K_w,\Q_p] = 2$. So from condition (b) again, there is a $p$-adic place $v$ of $L_1$ so that $[L_{1,v},\Q_p] = 2$, hence $[L_1,\Q] \geq 2$ and $L_1 \cong K$.

To show (iii), we first show that each $A_j \cong B$, and again without loss of generality will show it for $A_1$. After showing (ii), we now know that $A_1$ can also be considered as a quaternion algebra over $K$, so we need only to determine its ramification set. By assumption, $\mathrm{Ram}(B) = \{v_2,\omega\}$ where $v_2$ is a real place of $K$ and $\omega$ is a finite place such that it is the unique $q$-adic place of $K$ for some specific prime $q$. This means that for every prime $p$ other than $q$, each $p$-adic place $w$ of $K$ has $B \otimes_{K} K_{w}$ which is not a division algebra. From (b) above, this implies $A_1$ is not ramified at any $p$-adic place for any prime other than $q$. On the other hand, for each $q$-adic place $w$ of $K$, we know $B \otimes_{K} K_{w}$ is a division algebra. Hence $A_1$ is ramified at $\omega$, the unique $q$-adic place of $K$. Finally, because $A_1$ is ramified at only a single finite place of $K$, and $K$ has only two real places, and the set of ramified places must be even, we see that $\mathrm{Ram}(A_1) = \{r,\omega\}$ for one of the real places $r$ of $K$. So up to the previously mentioned ambiguity for such quaternion algebras over $K$, we have $A_1 \cong B$.

To finish (iii), we observe that from (c) we know that $\psi_j(\Sigma)$ is contained in the image of some order of $A_j \cong B$. Since $B$ has type number $1$, it has a unique maximal order up to conjugation, so after conjugating within $\PSL(2,\C)$ we may assume that $\psi_j(\Sigma) < \P\rho(\mathcal{O}^1)$ as desired.
\end{proof}

\section{Seifert fibred spaces}
\label{SFS_recap}
For the reader's convenience, we briefly recall some of the theory of Seifert fibred spaces, see \cite[Chapter VI]{J}, \cite{neumann} or \cite[\S 3]{Sc2} for more details.

\subsection{Some basic facts about Seifert fibred spaces}\label{ss:SFS}

A closed orientable $3$-manifold $M$ is a {\em Seifert fibred space} if and only if $M$ is foliated by circles.  The leaves of this foliation are either {\em regular fibres} or {\em exceptional fibres}. To describe this more carefully, 
we recall the following. 

Let $\mathbb{D}^2$ denote the unit disc. By a {\em trivially fibred solid torus} we mean $\mathbb{D}^2\times S^1$ with the product foliation by circles; i.e. the fibres are the circles $\{x\}\times S^1$.  A {\em fibred solid torus}
is a solid torus with a foliation by circles that is finitely covered by a trivially fibred solid torus.  Such a solid torus can be obtained from a trivially fibred solid torus by cutting it open along $\mathbb{D}^2 \times \{y\}$ for some $y\in S^1$ and gluing the ends of the solid cylinder by a $2q\pi/p$ twist where $p,q\in \Z$ are coprime. These integers can be normalized so that $0 < q < p$.

Returning to the Seifert fibred space $M$, a regular fibre is one which has a trivially fibred solid torus neighborhood, and otherwise it is an exceptional fibre, of which there are only finitely many.
Forming the quotient space of the Seifert fibred space by collapsing each fibre to a point, the result is a $2$-manifold with an orbifold structure (which we refer to as the base orbifold).  
When $\pi_1(M)$ is infinite with $t$ exceptional fibres and the base orbifold is orientable with underlying surface of genus $g$, then $\pi_1(M)$ has a presentation
$$\pi_1(M) = \<x_1,y_1,\ldots , x_g,y_g, c_1,\ldots , c_t, z \mid z~\hbox{is central},~c_j^{p_j}z^{q_j}=1,\ \prod [x_i,y_i]c_1c_2\ldots c_t = z^d\>,$$
where $0 < q_j < p_j$ 
for  $j=1, \ldots t$, with $q_j$ coprime to $p_j$,  and $d$ is an integer. The central generator $z$ is 
the homotopy class of any regular fibre. The set of pairs $\{(p_j,q_j) : j=1,\ldots ,t\}$ is called the {\em Seifert invariants} of $M$, and following \cite{neumann, Sc2}, we define 
the {\em Euler number}  $e(M)=-(d+\Sigma_{j=1}^{t} q_j/p_j)$. Note that when $g=0$, the case
of primary interest for us, 
the manifold is determined by the data $(e(M); (p_1,q_1),\ldots (p_t,q_t))$. 

The Seifert fibred space $M$ as above is geometric in the sense of Thurston (see \cite{Sc2}), and if the base is a hyperbolic $2$-orbifold (which will be the case of interest to us), then $M$ admits a $\mathbb{H}^2\times \R$ geometry or
$\widetilde{\PSL}_2$ geometry according to whether the Euler number of $M$ is zero or not. 
We also point out for future reference that if a Seifert fibred space $M$ with base a hyperbolic $2$-orbifold fibres over the circle, then $e(M)=0$ and $M$ has $\mathbb{H}^2\times \R$ geometry (see \cite[Theorem 5.4]{Sc2}).

\subsection{Finite extensions of Seifert fibred spaces}\label{ss:fin_ext}

It will be important for us to identify when central extensions of Fuchsian groups are the fundamental groups of Seifert fibred spaces, and to that end
we will make crucial use of the next result, which is \cite[Lemma 5.15]{GMW}. We include a sketch of the proof for completeness.

\begin{lemma}
\label{GMW_Lemma}
Suppose that $G$ fits into a short exact sequence 
$$1 \rightarrow \Z  \rightarrow G \rightarrow F \rightarrow 1,$$
where $F$ is a cocompact Fuchsian group. Then $G$ is the fundamental group of a Seifert fibred space if and only if $G$ is torsion-free.\end{lemma}

\begin{proof} If $G$ is the fundamental group of a Seifert fibred space $M$ with hyperbolic base, then $M$ is an aspherical manifold and so $G$ is torsion-free.

For the converse, let $p: G\rightarrow F$ denote the quotient homomorphism. We pass to a finite index torsion-free normal subgroup $F_0=\pi_1(\Sigma)$ of $F$  where $\Sigma$ is a closed, orientable, surface of genus at least $2$.  Let $K = p^{-1}(F_0)$ which fits into an induced short exact sequence
$$1 \rightarrow \Z  \rightarrow K \rightarrow F_0 \rightarrow 1.$$
$K$ is isomorphic to the fundamental group of an orientable circle bundle over $\Sigma$, which is Haken. The key point now is that the extension
$$1 \rightarrow K  \rightarrow G \rightarrow G/K \rightarrow 1$$
is {\em effective} in the language of \cite{GMW}, and as such one can apply a theorem of Zimmermann \cite[Satz 0.2]{Zim} to conclude that since $G$ is torsion-free, it is the fundamental group of a closed aspherical $3$-manifold, which is necessarily Seifert fibred since it contains an infinite cyclic normal subgroup.\end{proof}

\subsection{A result of Wilkes}
\label{ss:wilkes}

We now recall a result of Wilkes \cite{Wil} which identifies  precisely the class of Seifert fibred spaces
whose fundamental groups are profinitely rigid within the class of fundamental groups of   compact $3$-manifolds.

\begin{theorem}
\label{main_wilkes}
Let $M$ be a closed aspherical Seifert fibred space, and let $N$ be a compact $3$-manifold with $\wh{\pi_1(M)}\cong \wh{\pi_1(N)}$. Then, either $N\cong M$ or else $N$ is  a Seifert fibred space and both $M$ and $N$ are among 
the surface bundles over the circle with periodic monodromy that arise  in the construction of \cite{Hem}. In particular, when $\wh{\pi_1(M)}\cong \wh{\pi_1(N)}$ and $M$ and $N$ are not homeomorphic, both $M$ and $N$ have $\H^2\times\R$ geometry and $e(M)=e(N)=0$.
\end{theorem}

We shall use this in combination with the following lemma.

\begin{lemma}\label{l:dense-will-do}
Let $\La$ and $\G$ be the fundamental groups of Seifert fibred spaces over a hyperbolic base orbifold
that has fundamental group $\Delta$. Then, every homomorphism $j:\La\rightarrow \wh{\G}$
with dense image is injective and induces an isomorphism $\hat{j}: \wh{\La}\to\wh{\G}$
\end{lemma}

\begin{proof} We have short exact sequences $1\to \Z\to \La\to \Delta\to 1$
and $1\to \Z\to \G\to \Delta\to 1$. 
As  $j(\La)$ is dense in $\wh{\G}$, its centre   must map to the centre of $\wh{\G}$,
so $j$ induces a map $\pi:\Delta \to \wh{\Delta}$ with dense image such that the following diagram 
commutes.
$$ 
\begin{matrix}
&1\longrightarrow &\Z&\longrightarrow&
\La &\longrightarrow& \Delta&  \longrightarrow& 1
\cr
&&\ {\bigg\downarrow} &&{\bigg\downarrow} {j}
&& {\bigg\downarrow} {\pi}&&
\cr
&1\longrightarrow &\wh{\Z}&\longrightarrow&
\wh{\G} &\longrightarrow& \wh{\Delta} &  \longrightarrow& 1
\end{matrix}
$$ 
Since the image of $\pi$ is dense, it extends to an epimorphism $\wh{\pi}:\wh{\Delta}\to\wh{\Delta}$,
and since $\wh{\Delta}$ is Hopfian, this must be an isomorphism.  Thus, with an appeal 
to Corollary \ref{c:zhat},  we obtain a commutative
diagram 
$$ 
\begin{matrix}
&1\longrightarrow &\wh{\Z}&\longrightarrow&
\wh{\La} &\longrightarrow& \wh{\Delta}&  \longrightarrow& 1
\cr
&&\ {\bigg\downarrow} &&{\bigg\downarrow} {\hat{j}}
&& {\bigg\downarrow} {\wh{\pi}}&&
\cr
&1\longrightarrow &\wh{\Z}&\longrightarrow&
\wh{\G} &\longrightarrow& \wh{\Delta} &  \longrightarrow& 1
\end{matrix}
$$ 
with $\hat{j}$ a surjection and $\wh{\pi}$ an isomorphism. It follows that the restriction of $ {\hat{j}}$ to $\wh{\Z}$ is
an epimorphism. By the Hopf property for $\wh{\Z}$, this
restriction must be an isomorphism, and therefore $\wh{j}$ is an isomorphism. 
\end{proof}

\begin{corollary}\label{c:dense-enough}
Under the hypotheses of Lemma \ref{l:dense-will-do}, if the Euler number of either  Seifert fibred space
is non-zero then $\La\cong\G$.
\end{corollary}

\subsection{Some calculations}
\label{s:clacs}

We will  need  
information about the abelianizations of the fundamental groups of Seifert fibred spaces with base orbifolds $S^2(p,q,r)$ drawn from list (\ref{list-top}), 
and  control over  the Euler numbers of 
Seifert fibred spaces  with base from lists  (\ref{list-top}) and (\ref{list-bottom}). For convenience, we recall the standard presentation for the triangle group $\Delta(p,q,r)$:
$$\Delta(p,q,r) = \langle a,b,c | a^p=b^q=c^r=1, abc=1 \rangle.$$

Here and throughout this article, we write $C_n$ to denote the {\em cyclic group of order $n$.}
 
\begin{lemma}
\label{small_ab}
Let $M$ be a Seifert fibred space with base orbifold $S^2(p,q,r)$ where
$\D(p,q,r)$ is drawn from (\ref{list-top}),  
$$\Delta(3,3,4), \Delta(3,3,5), \Delta(3,3,6), \Delta(2,5,5), \Delta(4,4,4).$$
Then, $e(M)\neq 0$ and $H_1(M,\Z)$ is finite. In more detail, if 
$$
\pi_1(M) = \< c_1, c_2,c_3, z\mid z~\hbox{is central},~c_1^{p}z^{e_1}=c_2^{q}z^{e_2}=c_3^{r}z^{e_3}=1, c_1c_2c_3 = z^d\>
$$
then, in the various cases, $H_1(M,\Z)$ is:
\begin{enumerate}
\item $\Delta(3,3,4)$: $C_{3k_1}$ where $k_1= 4 e_1 + 4 e_2 + 3 e_3 + 12 d$
\item $\Delta(3,3,5)$: $C_{3k_2}$ where $k_2= 5 e_1 + 5 e_2 + 3 e_3 + 15 d$
\item $\Delta(3,3,6)$: $C_3\times C_{3k_3}$  where $k_3= 2 e_1 + 2 e_2 + e_3 + 6 d$
\item $\Delta(2,5,5)$: $C_{5k_4}$  where $k_4= 5 e_1 + 2 e_2 + 2 e_3 + 10 d$ 
\item $\Delta(4,4,4)$: $C_4\times C_{4k_5}$  where $k_5= e_1 + e_2 + e_3 + 4 d$. 
\end{enumerate}
\end{lemma}

\begin{proof}   
As $M$ is a Seifert fibred space over $S^2(p,q,r)$, it  has Heegaard genus $2$ (see \cite{BCZ}), and hence 
$\pi_1(M)$ has rank $2$. It follows that if $H_1(M,\Z)$ is not cyclic, then $H_1(M,\Z)\cong C_m\times C_n$.
The abelianizations of the five triangle  groups that we are considering are, in order,   $C_3$, $C_3$, $C_3\times C_3$, $C_5$ and $C_4\times C_4$, and the abelianization of $\pi_1M$ surjects the abelianization of its base.
Abelianizing the fundamental group of $\pi_1M$  in the five cases gives the relation matrices shown below:
$$\left(
\begin{array}{cccc}
 e_1 & 3 & 0 & 0 \\
 e_2 & 0 & 3 & 0 \\
 e_3 & 0 & 0 & 4 \\
 -d & 1 & 1 & 1 \\
\end{array}
\right),
\left(
\begin{array}{cccc}
e_1 & 3 & 0 & 0 \\
e_2 & 0 & 3 & 0 \\
e_3 & 0 & 0 & 5 \\
 -d & 1 & 1 & 1 \\
\end{array}
\right),
\left(
\begin{array}{cccc}
e_1 & 3 & 0 & 0 \\
e_2 & 0 & 3 & 0 \\
e_3 & 0 & 0 & 6 \\
 -d & 1 & 1 & 1 \\
\end{array}
\right),
\left(
\begin{array}{cccc}
e_1 & 2 & 0 & 0 \\
e_2 & 0 & 5 & 0 \\
e_3 & 0 & 0 & 5 \\
 -d & 1 & 1 & 1 \\
\end{array}
\right),
\left(
\begin{array}{cccc}
e_1 & 4 & 0 & 0 \\
e_2 & 0 & 4 & 0 \\
e_3 & 0 & 0 & 4 \\
 -d & 1 & 1 & 1 \\
\end{array}
\right).
$$
These have determinants $3 (4 e_1 + 4 e_2 + 3 e_3 + 12 d)$, 
$3 (5 e_1 + 5 e_2 + 3 e_3 + 15 d)$, 
$9 (2 e_1 + 2 e_2 + e_3 + 6 d)$, 
$5 (5 e_1 + 2 e_2 + 2 e_3 + 10 d)$ and 
$16 (e_1 + e_2 + e_3 + 4 d)$ respectively.

For $H_1(M,\Z)$ to be infinite these determinants must vanish.  Setting these determinants to zero, and rearranging, we see that in the case of the first four equations it follows that $4|e_3$, $5|e_3$ $2|e_3$ and $2|e_1$.  However,
by definition of the Seifert invariants (the coprimeness of $p$, $q$, $r$ with the $e_i$), none of these 
conditions is satisfied. For the final case, by definition of the Seifert invariants, each $e_i=1$ or $3$, and so $e_1 + e_2 + e_3$ is always odd, and hence the determinant cannot vanish in this case either.
It follows that these determinants are the orders of the abelianizations, as claimed.

As noted, for $M$ a Seifert fibred space as above, $\pi_1(M)$ has rank $2$, and so the structure of the abelianizations in cases (3) and (5) of Lemma \ref{small_ab} follows immediately. We also note that using the coprimeness condition
on the Seifert invariants in these cases that $k_3$ is coprime to $3$ and $k_5$ is coprime to $2$.

For the remaining cases of Lemma \ref{small_ab}, we note that if $Z$ denotes the
image of the centre in $H_1(M,\Z)$ then we have an exact sequence (where $C\cong C_3$ or $C\cong C_5$ is the abelianization of the triangle group base) 
$$1 \rightarrow Z \rightarrow H_1(M,\Z) \rightarrow C \rightarrow 1.$$
Once again, using the coprimeness condition
on the Seifert invariants we deduce that $(4 e_1 + 4 e_2 + 3 e_3 + 12 d)$ and $(5 e_1 + 5 e_2 + 3 e_3 + 15 d)$ are coprime to $3$ and $(5 e_1 + 5 e_2 + 3 e_3 + 15 d)$ is coprime to $5$.
Hence the order of $Z$  is coprime
to $|C|$, which is $3$ or $5$, and it follows that the abelianizations in these three cases are as claimed. 

That $e(M)\neq 0$ now follows easily since it is known (see \cite[Theorem 4.1]{Hem} or \cite[Theorem 5.4]{Sc2}) that if $M$ is a closed 
orientable Seifert fibred space with orientable base then $M$ is also a surface bundle over the circle with periodic monodromy if and only if $e(M)=0$. In our case, since $H_1(M,\Z)$ is finite, it cannot be a surface bundle over the circle and therefore $e(M)\neq 0$. \end{proof}

\begin{lemma}
\label{l:euler_no_bottomlist}
Let $M$ be a Seifert fibred space with base orbifold $S^2(p,q,r)$ where
$\D(p,q,r)$ is  from (\ref{list-bottom}), i.e. is one of
$$\D(2,3,8), \D(2,3,10), \D(2,3,12), \D(2,4,5), \D(2,4,8).$$
Then $e(M)\neq 0$.
\end{lemma}

\begin{proof} Each triangle group in the statement of Lemma \ref{l:euler_no_bottomlist} has  a subgroup of index
$2$, namely the corresponding triangle group from list (\ref{list-top}). Thus $M$ has a double cover $M'$ which is a Seifert fibred space with base $S^2(p',q',r')$ from (\ref{list-top}). For Euler numbers we have $e(M')=2 e(M)$
 (see \cite[Theorem 3.6]{Sc2} for example),  and 
 Lemma \ref{small_ab} tells us that $e(M')\neq 0$. \end{proof}

\section{Profinite rigidity of certain Seifert fibred spaces}
\label{SFS_section_rigid}

For the convenience of the reader,
we begin by recalling from \cite{BMRS2}  
that the triangle groups on which we are focussing are profinitely rigid.
\begin{theorem}\label{main_triangle}
Each of the   triangle groups in (\ref{list-top}) and (\ref{list-bottom}) is arithmetic and profinitely rigid.
\end{theorem}

\subsection{Reducing to the relative case}
\label{reduce_to_rel_rigid}

Let $\Delta = \Delta(p,q,r)$ be one of the triangle groups from Theorem \ref{main_triangle}, let $M$ be a Seifert fibred space with base $S^2(p,q,r)$, and $\Gamma = \pi_1M$. As a first step towards proving that $\Gamma$ is profinitely rigid, we reduce  to a ``relative situation" by proving the following result.

\begin{theorem} \label{reduce_seif}
For $\Gamma$ as above and $\Lambda$ a finitely generated, residually finite group with $\widehat{\Gamma} \cong \widehat{\Lambda}$,  there is a Seifert fibred space $N$ with base $S^2(p,q,r)$ such that $\Lambda \cong \pi_1(N)$.
\end{theorem} 

The main idea here is that the Galois rigidity of $\Delta$ can be upgraded to Galois rigidity for $\Gamma$.   Arguments similar to those in \cite{BMRS2} can then be used to show that $\Lambda$ must itself  
surject onto $\Delta$, giving $\Lambda$ as a $\Z$-central extension of $\Delta$. 

\begin{lemma}
 \label{galois_rigid_seif}
If $\Delta$ is Galois rigid, then $\G$ is Galois rigid with the same associated trace-field and quaternion algebra as $\Delta$.
\end{lemma}

\begin{proof} We have a surjection $p:\G\to\Delta$. Each Zariski-dense representation $\Delta\to \PSL(2,\C)$ pulls back to a Zariski-dense representation
$\Gamma\to \PSL(2,\C)$.  Since $\Delta$ is assumed to be Galois rigid, it suffices to show that every Zariski-dense representation $f: \Gamma \to \PSL(2,\C)$ factors through $p$.
But  
this follows from Lemma \ref{l:thru-proj}(2).\end{proof}  

Following the strategy from \cite{BMRS2}, it will be convenient to first prove Theorem \ref{reduce_seif} for the a limited list of the bases $S^2(p,q,r)$. The remaining triangle groups in Theorem \ref{main_triangle} will be handled in Proposition \ref{up_reduce}. 
However, we first comment upon and correct some statements from \cite{BMRS2}. 

The reader may have observed that the  triangle groups in (\ref{list-top}) are a proper subset of those listed in \cite[Proposition 3.3]{BMRS2}  where Galois rigidity is asserted. This is because it was wrong to assert
that the proof applied to  $\Delta(3,5,5)$ and $\Delta(5,5,5)$; indeed these groups are {\em not Galois rigid}. 
In hindsight, this is clear for $\Delta(5,5,5)$ since it is an index-2 subgroup of $ \Delta(2,5,10)$ 
and $ \Delta(2,5,10)$ surjects the triangle group
$\Delta(2,5,5)$, which  has no subgroup of index-2. The additional characters that show  $\Delta(3,5,5)$ is not Galois
rigid are less obvious; they are described in the erratum to \cite{BMRS2}. 
The proof of  \cite[Proposition 3.3]{BMRS2} is valid for the groups in list  (\ref{list-top}).
 A very pleasing and clean account of the set of characters of all $\PSL(2,\C)$ representations of triangle groups is given in \cite[Proposition D]{AB} and  the Galois rigidity of the groups in (\ref{list-top}), proved in \cite{BMRS2}, can  be verified using this account.

\begin{proposition}
 \label{reduce_galois_rigid}
When $\Delta = \Delta(p,q,r)$ is a triangle group from list (\ref{list-top}), the conclusion of Theorem \ref{reduce_seif} holds.
\end{proposition}

\begin{proof} The preceding discussion tells us that each of these triangle groups is Galois rigid.
In addition, from \cite[\S 3.2]{BMRS2}, the invariant trace-fields and quaternion algebras of these triangle groups all satisfy the hypothesis described in the preamble to Theorem \ref{galois_rigid_reps}. 
Thus, Lemma \ref{galois_rigid_seif} shows that $\Gamma$ is Galois rigid with the same invariant trace-field and quaternion algebra as $\Delta$.  

We now imitate the proof of \cite[Theorem 1.1]{BMRS2}, for which we first assume that $\Delta$ is one of the first four triangle groups in (\ref{list-top}). 
From \cite[Theorem 3.2]{BMRS2}, each of these triangle groups is exactly the image in $\PSL(2,\C)$ of the group of norm 1 elements in the unique (up to conjugation) maximal order of their quaternion algebra. In this context, Theorem \ref{galois_rigid_reps} provides a  finitely generated Zariski-dense  
 subgroup $L < \Delta$ and a surjection $\rho: \Lambda \twoheadrightarrow  L$ inducing a continuous surjection $\widehat{\rho}: \widehat{\Lambda} \twoheadrightarrow  \widehat{L}$.

Because $\G$ is a central extension $1 \to \Z \to \G \to \Delta \to 1$, Corollary \ref{exact} shows that $\wh{\Lambda} \cong \wh{\G}$ is also a central extension $1 \to \wh{\Z} \to \wh{\Lambda} \to \wh{\Delta} \to 1$, and so the kernel of the surjection $f:\wh{\Lambda} \to \wh{\Delta}$ is central. Since $L$ is Fuchsian, $\widehat{L}$ is centre-free by Theorem \ref{prof-centre-fuchs}, so that $\wh{\rho}$ descends along $f$ to give a surjection $h: \widehat{\Delta} \twoheadrightarrow \widehat{L}$ with $\widehat{\rho} = h \circ f$.  
Appealing to \cite[Theorem 5.1]{BMRS2} and  \cite[Theorem B]{BRprasad}, we deduce that $L = \Delta$, and by the Hopfian property for finitely generated profinite groups, $h: \widehat{\Delta} \twoheadrightarrow \widehat{L} = \widehat{\Delta}$ must also be injective. Since $h$ is an isomorphism and $\widehat{\rho} = h \circ f$, we can see that $\widehat{\rho}$ has the same kernel as $f$, which we know is central in $\wh{\Lambda}$. Thus the kernel of $\rho$ itself is central in $\Lambda$.

So far we have shown that $\rho$ gives a surjection of $\Lambda$ onto $\Delta$ with central kernel. It now follows from Proposition \ref{p:sameZ} that the center of $\Lambda$ must be isomorphic to that of $\G$, namely $\Z$. Since $\Gamma$ is torsion free, Theorems \ref{good-groups} and \ref{good-tors-free} imply that $\Lambda$ is also torsion-free. Putting this together with Lemma \ref{GMW_Lemma} implies that $\Lambda$ is isomorphic to the fundamental group of a Seifert fibred space, whose base has orbifold group isomorphic to $\Delta$. This completes the proof for the first four triangle groups in (\ref{list-top}).

We next consider the case of $\Delta=\Delta(4,4,4)$.
Now, $\Delta\subset \Delta(3,3,4)$ is a normal subgroup of index $3$, and in this
case the methods of \cite{BMRS2} only provide a surjection $\rho: \Lambda \to L$ and {\em a priori} we only know $L < \Delta(3,3,4)$. As above, we can construct $\widehat{\rho}: \widehat{\Lambda} \twoheadrightarrow  \widehat{L}$ and $h: \widehat{\Delta} \twoheadrightarrow \widehat{L}$ with $\widehat{\rho} = h \circ f$.
However, as in \cite{BMRS2},  if $L$ were not contained in $\Delta$, then $L$ would have a $\mathbb{Z}/3\mathbb{Z}$ quotient which would pull back to $\Delta$ along $h$, and $\Delta$ has no such quotient. Thus we see that $L < \Delta$, and the rest of the proof now proceeds as in the previous cases: we conclude that $L = \Delta$, that $\Lambda$ is a central extension of $\Delta$ by $\Z$, and that $\Lambda$ is isomorphic to the fundamental group of a Seifert fibred space with the appropriate base. \end{proof}

The triangle groups in (\ref{list-bottom}) can now be handled as in \cite{BMRS2} by understanding the index-$2$ extensions of the cases already proven in Proposition \ref{reduce_galois_rigid}.

\begin{proposition}
 \label{up_reduce}
When $\Delta = \Delta(p,q,r)$ is a triangle group in (\ref{list-bottom}), the conclusion of Theorem \ref{reduce_seif} holds.
\end{proposition}

\begin{proof}
Let $\Delta$ denote any one of the triangle groups (\ref{list-bottom}), and let $\Delta'$ be the unique triangle group appearing in (\ref{list-top}) which is an index-$2$ subgroup of $\Delta$.  
Let $\Gamma' < \Gamma$ be the index-$2$ subgroup whose image in $\Delta$ is $\Delta'$. From the isomorphism $\wh{\Lambda} \cong \wh{\Gamma}$ we obtain a corresponding index-$2$ subgroup $\Lambda' < \Lambda$ so that $\wh{\Lambda'} \cong \wh{\Gamma'}$. In particular, $\wh{\Lambda'}$ contains the central $\wh{\Z}$ which is the kernel of the surjection $\wh{\Lambda} \twoheadrightarrow \wh{\Delta}$ and the image of $\wh{\Lambda'}$ under this surjection is $\wh{\Delta'}$.

Now $\Gamma'$ is the fundamental group of a Seifert fibred space whose base is the triangle orbifold corresponding to $\Delta'$, and we may apply Theorem \ref{reduce_galois_rigid} to see that $\Lambda'$ is also isomorphic to the fundamental group of such a Seifert fibred space. In particular, $\Lambda'$ is a central $\Z$-extension of $\Delta'$. Because $\wh{\Delta}$ is centre-free, by Theorem \ref{prof-centre-fuchs}, $\Lambda'$ contains the centre of $\Lambda$. Hence $\Lambda$ is a central $\Z$-extension of some group, $G = \Lambda/ \Z$, which is an index-$2$ extension of $\Delta'$, and   $\wh{G} \cong \wh{\Delta}$. By \cite[Lemma 4.3]{BMRS2}, this implies $G \cong \Delta$, so   $\Lambda$ is also a central $\Z$ extension of $\Delta$.  Just as in the end of the proof of Proposition \ref{reduce_galois_rigid}, we can now apply Theorem \ref{good-groups} and Theorem \ref{good-tors-free} to show that $\Lambda$ torsion-free and Lemma \ref{GMW_Lemma} to show that $\Lambda$ is isomorphic to the fundamental group of a Seifert fibred space as claimed. This finishes the proof of Theorem \ref{reduce_seif}.
\end{proof}

\subsection{Completing the proof of profinite rigidity}
\label{ss_rigid}

In this subsection we put together the results from the previous sections to prove Theorem \ref{SFS_rigid}, which we restate for the reader's conveneince.

\begin{theorem}
\label{SFS_prorigid}
Let $M$ be a Seifert fibred space with base orbifold $S^2(p,q,r)$ associated to a triangle group in (\ref{list-top}) or (\ref{list-bottom}). Then $\G=\pi_1M$ is profinitely rigid.
\end{theorem}

\begin{proof} Suppose that $\Lambda$ is a finitely generated, residually finite group with $\wh{\Lambda}\cong \wh{\G}$. We know from Theorem \ref{reduce_seif} that $\Lambda$ is isomorphic to the fundamental group of a Seifert fibred space with the  same base $S^2(p,q,r)$ as $M$. From Lemmas \ref{small_ab} and \ref{l:euler_no_bottomlist} we
know $e(M)\neq 0$, so we can apply Theorem \ref{main_wilkes} to deduce that $\Lambda \cong\G$.
\end{proof}

\section{Reducing to Grothendieck rigidity: proof of Theorem \ref{t3}}\label{s:t3}

We remind the reader of the statement of Theorem \ref{t3} from the introduction.  

\begin{theorem} \label{second-t3}
Let $\Delta(p,q,r)$ be a triangle group from (\ref{list-top}) or (\ref{list-bottom}), let $M$ be a 
Seifert fibred space with  base orbifold $S^2(p,q,r)$ and let $\G=\pi_1M$.
Then, for  every finitely generated, residually
 finite group $\L$ with $\wh{\L}\cong\wh{\G\times\G}$, there is
 an embedding $\L\hookrightarrow\G\times\G$ that induces the isomorphism.
 \end{theorem}

We begin by establishing some notation that will enable us to describe the outline of the proof.
Given a  finitely generated, residually finite group $\L$ with $\wh{\L}\cong \wh{\G\times \G}$,
we fix an isomorphism and 
use the inclusion  $\Lambda\hookrightarrow \wh{\L}\cong \wh{\G\times \G}$ to identify $\L$ with its image in
$\wh{\G\times \G}$.  
To lighten the notation, we write $\wh{\G}_1=\wh{\G}\times 1$ and $\wh{\G}_2=1\times \wh{\G}$.
Let $\L_i$ ($i=1,2$) denote the projections of $\Lambda$ to $\wh{\G}_1$ and $\wh{\G}_2$ respectively.
Our ultimate goal, for suitable $\G$, is to prove that $\L_1\cong\L_2\cong\G$ and that
$\L\hookrightarrow\L_1\times\L_2$ is a Grothendieck pair.  

The main arguments in this section build on Theorem \ref{products_prof_thm} to establish the following proposition. This covers the $\D(p,q,r)$ drawn from (\ref{list-top}); we deal with (\ref{list-bottom}) at the end of this section.

\begin{proposition}\label{p:towardsGP} Let $\G$ be the fundamental group of a \SFS whose base orbifold
is $S^2(p,q,r)$ with $\D(p,q,r)$ from (\ref{list-top}) and let 
$\L$ be as above. Then, 
\begin{enumerate}
\item $\L_i$ is dense in $\wh{\G}_i$, for $i=1, 2$;
\item $\L_i$ is the fundamental group of a Seifert fibred space with base $S^2(p,q,r)$.
\end{enumerate}
\end{proposition}

Before proving this proposition, we explain why it implies Theorem \ref{t3} for those $\D(p,q,r)$ from (\ref{list-top}). \\[\baselineskip]
{\noindent{\bf Proof of Theorem \ref{t3} (=\ref{second-t3}) for those $\D(p,q,r)$ from (\ref{list-top}).} With the notation established above,
Proposition \ref{p:towardsGP} allows us to apply Corollary \ref{c:dense-enough} and  conclude that
 $\La_1\cong\La_2\cong \G$. 
 By construction, the inclusion $\L\hookrightarrow\wh{\G\times\G}$,
  which realises the given isomorphism $\wh{\L}\cong\wh{\G\times\G}$, factors
  through $\L\hookrightarrow\L_1\times\L_2$. And since  
$\L_1\cong\L_2\cong\G$, the
  Hopf property for $\wh{\G\times\G}$ implies that there is an automorphism
  of $\wh{\G\times\G}$  taking $\L_1\times\L_2$ to $\G\times\G$.
  \qed

\subsection{Proof of Proposition \ref{p:towardsGP}}\label{s:towardsGP}
 
 We continue with the notation established above: $\G$ will be the fundamental group of a \SFS over a 
base with fundamental group $\D$, 
\begin{equation}\label{sfs-seq}
 1\to \Z \to \G \to \Delta \to 1,
\end{equation}
 and $\L_i \subset \wh{\G}_i$ is the image of $\L<\wh{\G}\times \wh{\G}$ under coordinate projection.  Item (1)   of Proposition \ref{p:towardsGP} is obvious,
 given these definitions, and our task is to prove Proposition \ref{p:towardsGP}(2).
 
 We now define $N_i= \L\cap\wh{\G}_i$, 
 so  $\L_1 = \L/N_2$ and  $\L_2 = \L/N_1$, and we define $L_i<\wh{\Delta}$ to be the image of 
 $\L_i$ under $\wh{\G}_i\to\wh{\Delta}$, where $\G\times\G\to \D\times\D$ is the product
 of two copies of the map from (\ref{sfs-seq}).

 \begin{lemma}\label{l:split}
 Every representation  $\rho : \L\to {\rm{PSL}}(2,\C)$ with Zariski-dense image factors  
 through one of the projections $\L\to \L_1$ or  $\L\to \L_2$.
 \end{lemma}
 
 \begin{proof} Apply Lemma \ref{l:thru-proj}(1) with $H_i=N_i = \L\cap \wh{\G_i}$ to see that one of the $N_i$ must map   trivially, hence $\rho$ factors through $\L\to\L_{j\ne i}$.
 \end{proof}

 We also need a profinite analogue of Lemma \ref{l:thru-proj}.

\begin{lemma}\label{l:justZ} Let  $\Delta$ be  a non-elementary Fuchsian group and let $G_1$ and $G_2$ be 
finitely generated profinite groups. 
\begin{enumerate}
\item For every continuous epimorphism $\phi:G_1\times G_2\twoheadrightarrow\wh{\Delta}$, either $\phi(G_1)=1$
or $\phi(G_2)=1$.
\item If
$\wh{\Delta}\cong G_1/Z$ with $Z<G_1$ central, then $Z$ is the centre of $G_1$ and
$\ker\psi = Z$ for every epimorphism  $\psi: G_1\to \wh{\Delta}$.
\end{enumerate}
\end{lemma}

\begin{proof} First we prove (1). The embedding  $\Delta\to \wh{\Delta}$ induces a bijection between the conjugacy classes of finite
subgroups \cite[Theorem 5.1]{BCR}.
 In particular, since $\Delta$ does not contain a finite normal subgroup, neither does $\wh{\Delta}$,
 so if the assertion of (1) were false, then both  $\phi(G_1)$
and $\phi(G_2)$ would be infinite. Passing to a subgroup of finite index in $\Delta$ and replacing 
$G_1$ and $G_2$ by the pre-images of this subgroup reduces us to the case where $\Delta$ is torsion-free,
i.e. is a free or surface group. But a non-trivial, finitely generated, closed normal subgroup of the profinite
completion of such a group (indeed of any limit group \cite[Theorem B]{Gut}) is of finite index. But  $\phi(G_1)$
and $\phi(G_2)$ commute, so cannot both be of finite index, otherwise, $\phi(G_1)\cap \phi(G_2)$ is an abelian subgroup of finite index in $\wh{\Delta}$, which is a contradiction. 

For (2), from Theorem \ref{prof-centre-fuchs} we know that $\widehat{\D}$ is centreless, hence $Z=Z(G_1)$ and
$Z\le \ker\psi$ for all epimorphisms $\psi:G_1\twoheadrightarrow\widehat{\D}$. In particular, $\psi$ factors
through $G_1/Z$ and the induced map $\overline{\psi}: G_1/Z\twoheadrightarrow\widehat{\D}$ gives rise to an
epimorphism 
$$\widehat{\D}\overset{\cong}\to G_1/Z\overset{\overline{\psi}}\twoheadrightarrow\widehat{\D}.$$
As $\widehat{\Delta}$ is Hopfian, we deduce that $\overline{\psi}$ is injective, hence $\ker\psi = Z$.
\end{proof}

 \begin{proposition}\label{p:neat}
 Let $\D$ be a hyperbolic triangle group and let $\G$ be a finitely generated central extension of $\Delta$.
 If $\L$ is a finitely generated group with $\wh{\L}\cong\wh{\G\times\G}$, then there does not exist a
 homomorphism $h :  \L\to \D$ with image a Zariski-dense proper subgroup.
 \end{proposition}
 
\begin{proof} Note that since the image of $h$ is contained in the discrete group $\D$, Zariski-dense and non-elementary are the same. Thus, if there were such a  homomorphism $h$, with image $S$ say, then Lemma \ref{l:justZ}(1) would tell us
that $\wh{h}$ factored through one of the coordinate projections in $\wh{\G\times\G}$.
And since, by Theorem \ref{prof-centre-fuchs}, the centre of $\wh{S}$ is trivial,
$\wh{h}$ would actually  factor through $\wh{\G}$ modulo
its centre, which is $\wh{\Delta}$. But from \cite[Theorem 5.1]{BMRS2} and  \cite[Theorem B]{BRprasad} we know that 
there is no epimorphism from $\wh{\Delta}$ to $\wh{S}$.
\end{proof}

We now  assume that $\D=\D(p,q,r)$ is one of the triangle groups from list (\ref{list-top}), the key feature of
these groups is again their Galois rigidity and specific arithmetic properties which ensure Theorem \ref{products_prof_thm} can be applied.

\begin{proposition} \label{p:sdp} Suppose that $\Delta=\Delta(p,q,r)$ is from list (\ref{list-top}) and 
that $\G$ is the fundamental  group of a \SFS whose base orbifold is $S^2(p,q,r)$. If
$\L$ is a finitely generated, residually finite group with $\wh{\L}\cong\wh{\G\times \G}$
and $\L_i$ is its projection to $\wh{\G}_i$, then
there exist epimorphisms $g_i:\L_i\to\Delta$ and hence a homomorphism
$$
g: \L \hookrightarrow \L_1\times\L_2 \overset{(g_1,g_2)}\to \Delta\times\Delta
$$
with image a full subdirect product.
\end{proposition}

\begin{proof}
Arguing as in the first paragraph of the proof of Proposition \ref{reduce_galois_rigid}, we see that  $\G$ satisfies the hypotheses of Theorem \ref{products_prof_thm}. Thus, from Theorem \ref{products_prof_thm}(i) we obtain representations 
$\psi_i: \L \to  \P\rho(\mathcal{O}^1) < {\rm{PSL}}(2,\C)$ for $i=1,2$, with $\psi_1$ and $\psi_2$
 distinct up to conjugation in 
${\rm{PSL}}(2,\C)$. Crucially, $\Delta = \P\rho(\mathcal{O}^1)$ for all of the triangle groups $\D$ in list (\ref{list-top}) except $\D(4,4,4)$, so $\psi_i$ has image in $\Delta$. An additional argument is required in the case $\D(4,4,4)$,  which 
is a normal subgroup of index $3$ in the corresponding
unit group $ \P\rho(\mathcal{O}^1) = \D(3,3,4)$. In this case, if  
$I:=\im \psi_i$ were not contained in $\D(4,4,4)$, then we would have a surjection  $\L\to I\to  C_3$ and hence 
$\wh{\L}\to \wh{I}\to C_3$. But this is not possible, because $\wh{I}$ has  trivial centre, by Theorem 
\ref{prof-centre-fuchs},
and $\wh{\L} = \wh{\G}\times\wh{\G}$ modulo its centre is
$\wh{\D}\times\wh{\D}$, where $\D=\D(4,4,4)$, and $\wh{\D}\times\wh{\D}$
does not map onto $C_3$ since  $\D(4,4,4)$ does not. 

We have argued that in all cases $\psi_i$ has image in $\Delta$. 
The images of these representations $\psi_i$ are Zariski-dense  
 and by Proposition \ref{p:neat} each  must be a surjection. Thus we obtain a homomorphism $\psi: \L \overset{(\psi_1,\psi_2)}\to \Delta\times\Delta$ whose image is a  subdirect product. We claim that in fact this is a {\em full} subdirect product, i.e. that
the image of $\psi=(\psi_1, \psi_2)$
intersects both factors of $\Delta\times\D$. To see why this is true,
consider what would happen if $\im \psi$
were to intersect $\D\times 1$ trivially. In this case,
the restriction of the projection $\D\times\D\to 1\times\D$ to $\im \psi$  
would be an isomorphism with inverse $\psi_2(\lambda)\mapsto (\psi_1(\lambda), \psi_2(\lambda))$. The composition
of this last map with the  projection $\D\times\D\to \D\times 1$ would give a surjection
$\psi_2(\lambda)\mapsto \psi_1(\lambda)$ from $\D$ to $\D$. As $\D$ is Hopfian, this surjection must be an automorphism.
However, $\D$ is rigid, in the sense that the faithful discrete representation of $\D$ in $\PSL(2,\C)$ is unique up to conjugacy in $\PSL(2,\C)$, so that
this automorphism of $\D$ must be the restriction of a conjugation in $\PSL(2,\C)$.
But there is no such conjugation, because  $\psi_1$ and $\psi_2$ are distinct up to conjugation in 
$\PSL(2,\C)$,  hence there can be no such automorphism.

Lemma \ref{l:split} tells us that each $\psi_i$ must factor through at least one of the maps $\L\to\L_{j(i)}$,
while Lemma \ref{l:justZ}(1) tells us that each  $\wh{\psi}_i$
must factor through exactly one of the maps $\wh{\L}\to\wh{\G}_j$.
We shall complete the proof by arguing that we can choose $\{j(1), j(2)\}=\{1,2\}$. It is clear 
that if $\wh{\psi}_i$ factors through $\wh{\L}\to\wh{\G}_j$ then $\psi_i$ factors
through $\L\to\L_{j}$, so we will be done if we can
derive a contradiction from the assumption that $\wh{\psi}_1$ and $\wh{\psi}_2$ both
factor through $\wh{\L}\to\wh{\G}_1$. 
Suppose, then, that  for $i=1,2$ we have an epimorphism
$\wh{\Psi}_i: \wh{\G}_1\to\wh{\D}$ so that $\wh{\psi}_i=\wh{\Psi}_i\circ p$ where $p:\wh{\L} \to\wh{\G}_1$ 
is the coordinate projection. 
Lemma \ref{l:justZ}(2) tells us that $\ker\wh{\Psi}_1=\ker\wh{\Psi}_2=Z(\wh{\G}_1)$.
Therefore,  
$\im(\wh{\psi}_1, \wh{\psi}_2) = \im(\wh{\Psi}_1, \wh{\Psi}_2) < \wh{\D}\times\wh{\D}$ does not intersect either factor,
contradicting the fact that the image of $(\psi_1, \psi_2)$ in $\Delta\times\D$ intersects both factors.

The map $g_{j(i)}$ in the statement of the proposition is the restriction to $\L_i$ of the
map $\wh{\G}_{j(i)}\to\D$ defined above.
\end{proof}

We have constructed, for $i=1,2$, an epimorphism $\L_i\to \D$ that 
extends to a continuous epimorphism $\wh{\G}_i\to \wh{\D}$. Lemma \ref{l:justZ}(2) tells us that the 
kernel of this second map is central, so we get a short exact sequence
\begin{equation}\label{sequ}
1\to Z \to \L_i \to \D\to 1
\end{equation} 
with $Z$ central in $\L_i$.

\begin{lemma}\label{inf_cyclic}
$Z$ is infinite cyclic.
\end{lemma}

\begin{proof}
Associated to the short exact sequence we have the following 5-term exact sequence in homology, noting that  $Z=H_0(\D,Z)$ since $Z$ is central.
$$
H_2(\D,\Z)\to Z\to H_1(\L_i,\Z)\to H_1(\D,\Z)\to 0.
$$
$Z$ is torsion-free,  $H_2(\D,\Z)\cong\Z$ and $H_1(\L_i,\Z)$ is finite, since it is  a quotient of $H_1(\L,\Z)\cong H_1(\G\times \G,\Z)$. Therefore $Z\cong\Z$.
\end{proof}

\noindent{\bf{End of the proof of Proposition \ref{p:towardsGP}}:}
With the short exact sequence (\ref{sequ}) and Lemma \ref{inf_cyclic} in hand, all that remains to be observed is that 
$\L_i$ is torsion-free, which follows from $\wh{\G}_i$ being torsion-free; this being a consequence of the goodness of $\Gamma$ (Theorems \ref{good1} and \ref{good2}). Hence 
we can appeal to Lemma \ref{GMW_Lemma} and the proof is complete. \qed\\[\baselineskip]
To complete the proof of Theorem \ref{second-t3} we deal with those Seifert fibred spaces with base orbifolds associated to triangle groups from (\ref{list-bottom}).

\begin{proposition}
\label{p:deal_with_bottom}
 Let $\Delta=\Delta(p,q,r)$ be a triangle group from (\ref{list-bottom}), let $M$ be a 
Seifert fibred space with  base orbifold $S^2(p,q,r)$ and let $\G=\pi_1M$.
Then, for  every finitely generated, residually finite group $\L$ with $\wh{\L}\cong\wh{\G\times\G}$, there is an embedding $\L\hookrightarrow\G\times\G$ that induces the isomorphism.
\end{proposition}

\begin{proof} For a group $\Delta$ from (\ref{list-bottom}) we can find a group $\Delta'$ from (\ref{list-top}) with $[\Delta : \Delta']=2$, and hence a Seifert fibred space $M'$ with fundamental group $\G'$ and $[\G : \G']=2$. 
Note that $\G\times\G$ contains $\G'\times\G'$ as a normal subgroup of index-$4$,  
and hence with $\L$ as in
the statement of Proposition \ref{p:deal_with_bottom}, we also have an index-$4$ normal subgroup $\L' < \L$ with $\wh{\L}'\cong \wh{\G'\times \G'}$. 

Now from what we have already proved about the groups from (\ref{list-top}), it follows that $\L'\hookrightarrow \G'\times \G'$ is a Grothendieck pair. Let $p_1: \wh{\G\times \G}\rightarrow \wh{\G}\times 1$ and 
$p_2: \wh{\G\times \G}\rightarrow 1\times \wh{\G}$, and set $p_i(\L)=\L_i$ for $i=1,2$ (where we view $\L \subset \wh{\G\times \G}$). By Proposition \ref{p:towardsGP} and the proof of Theorem \ref{t3}, we have that $p_i(\L')=\G'$ for $i=1,2$.
Note that $\L_1$ and $\L_2$ are dense in each factor, and so properly contain the group $\G'$ of finite index. By Theorems \ref{good1} and \ref{good-tors-free}, $\wh{\G}$ is torsion-free and hence we may deduce from Lemma \ref{GMW_Lemma}
that $\L_1$ and $\L_2$ are the fundamental groups of Seifert fibred spaces with base $S^2(p,q,r)$.  It now follows from Lemma \ref{l:dense-will-do} and Corollary \ref{c:dense-enough} that $\L_1\cong\L_2\cong\G$, and we complete the argument as we did  in the proof of Theorem \ref{t3} (in the case of the triangle groups from (\ref{list-top})).\end{proof}

\section{No finitely presented Grothendieck pairs}\label{products}

In this section we prove a strengthened version of Theorem \ref{t:not-fp}:
there are no Grothendieck pairs $\L\hookrightarrow \G\times\dots\times \G$ with $\G$ the fundamental group
of a Seifert fibred space and
$\L$ a proper, finitely presented subgroup. We shall see that this follows from the corresponding
result for Fuchsian groups that comes from \cite{BW} and \cite{BHMS2}.

We will make use of the following lemma, whose proof is straightforward.
 
 \begin{lemma}\label{l:usual} Let $A$ and $B$ be finitely generated groups.
If $f:A\to B$ induces an isomorphism $\wh{A}\to\wh{B}$, and $B_0<B$ has finite index, then $f|_{A_0}$ induces an isomorphism $\wh{A}_0\to\wh{B}_0$, where $A_0=f^{-1}(B_0)$.
\end{lemma}

\begin{proposition}
\label{p:fills}
Let $\D$ be a non-elementary Fuchsian group, let $D$
be the direct product of finitely many copies of $\D$,
and let $\L<D$ be a subgroup such that the
inclusion $\L\hookrightarrow D$ induces an isomorphism of 
profinite completions $\wh{\L}\cong\wh{D}$. If $\L$
is finitely presented, then $\L= D$.
\end{proposition}

\begin{proof}  
This essentially follows \cite{BW}. The details are given below.
Let $S<\D$ be a torsion-free subgroup of finite index, let $D_0=S\times\dots\times S$
and let $\Lambda_0=\Lambda\cap D_0$. Note that $S$ is either a surface group or a 
finitely generated free group.
$\Lambda_0$ is finitely presented and 
the restriction of the inclusion $\iota:\L\to D$
induces an isomorphism $\hat\iota_0:\wh{\L}_0\cong\wh{D}_0$ (by Lemma \ref{l:usual}).
But in a direct product of finitely many free and surface groups (more generally, a product
of limit groups) all finitely presented subgroups
are closed in the profinite topology -- see \cite{BW} and \cite{BHMS2} -- so
$\hat{\iota}_0$ being an isomorphism implies $\L_0=D_0$. Thus $\L$ has finite
index in $D$, and since $\hat{\iota}$ is an isomorphism, we conclude that $\L=D$.
\end{proof}
 
Theorem \ref{t:not-fp}  from the  introduction is a special case of the following result.

\begin{theorem}\label{p:fills2} 
Let $M$ be a Seifert fibred space with hyperbolic base orbifold, let $\G=\pi_1M$,
let $D$ be the direct product of finitely many copies of $\G$, 
and let $\L<D$ be a subgroup such that the inclusion  induces an isomorphism of 
profinite completions. If $\L$
is finitely presented, then $\L= D$.
\end{theorem}

\begin{proof}   
If $\G$ is virtually a product (a surface or free group times $\Z$), since the product of surface groups, free groups and copies of $\Z$ is a limit group, as noted above, it follows that  every 
finitely presented subgroup of $D$ is closed in the profinite topology. Thus we may assume
that $M$ is not finitely covered by a product bundle, and therefore the base orbifold is closed.

We first project $\L$ to the product of the base groups $B=\Delta\times\dots\times \Delta$, and note that the image group is still finitely presented since the kernel is finitely generated
(being in the centre). 
Thus, as in the proof of Proposition \ref{p:fills}, 
we deduce this mapping is onto (since the image of $\Lambda$ will be dense in $\wh{B}$).

We now have an epimorphism $\wh{\Lambda}\rightarrow \wh{B}$. We claim that this implies that $\L$ contains the centre $\mathcal{Z}=\Z^r<D$ (where $r$ is the number of factors). 
To see this, first note that if $\Lambda \cap \mathcal{Z}$ has finite index in 
$\mathcal{Z}$, then $\Lambda \cap \mathcal{Z}=\mathcal{Z}$, because otherwise $[D:\Lambda]<\infty$, and this contradicts $\wh{\Lambda}= \wh{D}$.
With this observation, the only concern is that $\L$ 
meets one of the direct factors of $\mathcal{Z}$ trivially.  We claim that this cannot happen.  The reason is this: since $M$ is a Seifert fibred space, $\G$ is good by Theorem \ref{good-groups}, 
hence $D$ is also good by \cite[Proposition 3.4]{GJZ}, and so in particular the fact that
$D$ has cohomological dimension $3r$ is witnessed by the continuous cohomology of $\wh{\D}$ with finite field coefficients.  If $\L$ meets one of the factors of $\mathcal{Z}$ trivially, then $\L$ will be an extension of $B$ by $\Z^s$,
with $s<r$, and in particular will be good (using Theorem \ref{good-groups}, by another application of \cite[Proposition 3.4]{GJZ}).  
But, in this case $\L$ has cohomological dimension less than $3r$,   and therefore so does $\wh{\L}$, a contradiction. \end{proof} 
 
\def\L{\Lambda}
\def\D{\Delta}

\section{Constructing Grothendieck Pairs in Products of Central Extensions of Hyperbolic Groups}  
\label{s:homol-argument}

Our purpose in this section is to establish the following 
 criterion for showing that the (strong) profinite genus of certain direct products is infinite. In the next section
 we shall explain why this criterion applies to the fundamental groups of many Seifert fibred
 spaces whose base orbifold is a hyperbolic triangle group. 
 
\begin{theorem}\label{t2:lots-of-P}
Let $\Delta$ be a non-elementary hyperbolic group and
let $\G$ be a group with $H_2(\G,\Z)=0$ that maps onto $\Delta$.
Let $G$ be a finitely generated group that maps onto a subgroup of finite index in $[\G, \G]$. Then, 
\begin{enumerate}
\item there exists an infinite sequence of distinct finitely generated subgroups $P_n<G\times G$
such that each inclusion $u_n:P_n\hookrightarrow G\times G$  induces an isomorphism of profinite completions. 
\item If $G$ is a central extension of a hyperbolic group and centralizers of elements in that hyperbolic group are virtually cyclic, 
then $P_n$ is not abstractly isomorphic to $P_m$ when $n\neq m$.
\end{enumerate}
\end{theorem}

\noindent The subgroups $P_i$ will not be finitely presented in general, even if $G$ is finitely presented (cf.~\cite{B-jems}).

This criterion extends a well-established train of ideas, which we now explain.
Grothendieck \cite{Groth} asked if there exist Grothendieck pairs of finitely presented groups. This problem was eventually solved by
Bridson and Grunewald \cite{BG}. Their proof builds on an earlier argument of Platonov and Tavgen \cite{PT}
who constructed the first Grothendieck pair of finitely generated groups. They did this by appealing to a special case of 
the following proposition, taking $G$ to be a free group and $Q$ to be Higman's famous example
of a 4-generator, 4-relator group with $\widehat{Q}=1$ and $H_2(Q,\Z)=0$ (see for example \cite[Lemma 2.1 \& Theorem 5.1]{BG} for proofs of both statements).  

\begin{proposition}\label{p:PT}
Let $f:G\to Q$ be an epimorphism of groups, with $G$ finitely generated and $Q$ finitely presented. Consider the fibre product 
$P=\{(g,h) \mid f(g)=f(h)\} < G\times G$. Then,
\begin{enumerate}
\item $P$ is finitely generated;
\item if $\widehat{Q}=1$ and $H_2(Q,\Z)=0$, then $P\hookrightarrow G\times G$ induces an isomorphism $\widehat{P}\overset{\cong}\hookrightarrow\widehat{G\times G}$.
\end{enumerate}
\end{proposition} 

In the course of constructing counterexamples to the Platonov Conjecture, Bass and Lubotzky \cite{BL} described infinitely 
many Grothendieck pairs  by applying the full force
of Proposition \ref{p:PT}.
In order to do this, they needed  a technique for mapping hyperbolic groups onto finitely presented groups $Q$ with
$\wh{Q}=1$ and  $H_2(Q,\Z)=0$. This was achieved in two stages. First, a bespoke
small-cancellation argument due to Olshanskii \cite{Ol}  
shows that every non-elementary hyperbolic group $G$ maps onto a finitely presented group
with $\wh{Q}=1$, but one does not have control over $H_2(Q,\Z)$. If $G$ is free, this is not a serious problem -- one can replace 
$Q$ by its universal central extension $\widetilde{Q}$ and lift $G\to Q$ to a surjection $G\to\widetilde{Q}$ -- but for an arbitrary hyperbolic
group $G$ one cannot do this.  
Bass and Lubotzky got around this by using specific homological
properties of the groups that were their concern, with an argument that inspired Lemma \ref{l:schur} below.

Advances in the understanding of hyperbolic and related groups mean that
today one can replace Olshanskii's carefully crafted argument with a more flexible  and conceptually-easier construction. We shall use this
flexibility to prove that if  $\G$ is a non-elementary hyperbolic group with trivial 
centre, then there is an infinite sequence of non-isomorphic, finitely presented,
centerless groups $Q_i$ and  epimorphisms $G\to Q_i$ with ${Q}_i\not\cong Q_j$ if $i\neq j$ (Theorem \ref{t:lots-of-Q2} below).
The proof draws heavily on ideas from \cite{AMO} and \cite{abj}.

\subsection{Homology of groups and central extensions}
We assume that the reader is familiar with the homology of groups, as explained in \cite{brown}.  For
the theory of universal central extensions, the standard reference is \cite[pp. 43-47]{milnor}. We shall need only
the following basic facts.  The assertion in (3) is a consequence of (2): if $\widetilde{Q}$ mapped onto a finite
group $G$, then $Q=\widetilde{Q}/Z$ would map onto $G$ modulo its centre
 
 \begin{lemma}\label{l:univ} If $Q$ is a perfect group, then there is a central extension $1\to Z\to \widetilde{Q}\overset{p}\to Q\to 1$
 where
 \begin{enumerate}
 \item $\widetilde{Q}\overset{p}\to Q$ is universal: if the kernel of $E\overset{r}\onto Q$ is central in $E$, then there is a 
 unique homomorphism $f:\widetilde{Q}\to E$ such that $p=r\circ f$;
 \item $H_1(\widetilde{Q},\Z)=H_2(\widetilde{Q},\Z)=0$;
 \item if $Q$ has no non-trivial finite quotients, then neither does $\widetilde{Q}$.
 \end{enumerate}
 If $Q$ is finitely presented, then $H_2(Q,\Z)$ is finitely generated and $\widetilde{Q}$ is finitely presented. 
  \end{lemma}

 \begin{proof} The points other than (3) are covered in \cite[pp. 43-47]{milnor}.
 For (3), observe that if $\widetilde{Q}$ maps onto a finite
group $G$ then $Q=\widetilde{Q}/Z$ maps onto $G$ modulo its centre; since $Q$ has
no  non-trivial finite quotients, this means that $G$ is a quotient of $Z$, hence it is abelian; 
and from (2) we know that $\widetilde{Q}$ has no non-trivial abelian quotients.
 \end{proof}

\begin{lemma}\label{l:schur}
Let $Q$ be a perfect group with universal central extension $p:\widetilde{Q}\to Q$, let $G$ be a group with $H_2(G,\Z)=0$
and let $F:G\to Q$ be an epimorphism that restricts to $f:[G,G]\to Q$. Then, there exists an epimorphism $\tilde{f}:[G,G]\to\widetilde{Q}$
with $p\circ\tilde{f}=f$.
\end{lemma}

\begin{proof} The fibre product of the maps $F$ and $p$ is the subgroup $\breve{G}=\{(x,y)\mid F(x)=p(y)\}<G\times\tilde{Q}$.
By projecting to the first factor we see that $\breve{G}$ is a central extension
$$
0\to Z \to \breve{G} \to G\to 1,
$$
where $Z=\ker p \cong H_2(Q,\Z)$. 

Consider the standard 5-term exact sequence associated to this extension:
$$
H_2(G, \Z) \to Z \to H_1(\breve{G}, \Z) \to H_1(G, \Z) \to 0.
$$
The first term is zero by hypothesis, so $Z$ injects into $H_1(\breve{G}, \Z)$ -- in other words 
$Z \cap [\breve{G}, \breve{G}]$ is trivial. Thus $\breve{G}\onto G$ restricts to an
isomorphism  $[\breve{G}, \breve{G}]\to [G,G]$. By composing the inverse of this isomorphism
with the coordinate projection $\breve{G}\to 1\times \widetilde{Q}$ we obtain the desired 
map $\tilde{f}:[G,G]\to \tilde{Q}$, which is onto because its image contains $[\widetilde{Q}, \widetilde{Q}]$ and
$\widetilde{Q}$ is perfect.
\end{proof}
  
 \subsection{Common quotients of relatively hyperbolic groups}
The theory of relatively hyperbolic groups was outlined by Gromov \cite{gro} and developed by Farb \cite{Farb}, Bowditch \cite{bow} and others. Roughly speaking, a group is hyperbolic relative to a system of peripheral subgroups if it acts in a controlled manner 
on a Gromov-hyperbolic metric space with conjugates of the peripheral subgroups as isotropy subgroups.
The only examples that we shall need to consider are (i) hyperbolic groups, in which case the system of peripheral subgroups is trivial, and (ii) non-trivial free products $G_1\ast G_2$,
in which case the system of peripheral subgroups is $\{G_1,G_2\}$.  

Theorem 1.4 of \cite{AMO}, which is an application of Theorem 2.4 from \cite{Osin}, 
 states that every pair of properly relatively hyperbolic groups has a common quotient that is properly relatively hyperbolic, with control on the peripheral subgroups. Theorem \ref{t:osin} is a special case of this, with some adornments that
are implicit in  \cite{AMO}. We are grateful to Daniel Groves for a discussion of this result and  references.

\begin{theorem}[\cite{Osin}, \cite{AMO}]\label{t:osin} Let $G$ be a non-elementary hyperbolic group  and
let $H=H_1\ast H_2$ be a free product of non-trivial finitely presented groups that have trivial centre. 
Then there is an epimorphism $\mu: G\ast H\to Q$, such that
\begin{enumerate}
\item $Q$ is finitely presented;
\item the restriction of $\mu$ to each of $G$ and $H$ is surjective; 
\item the restriction of $\mu$ to $H_1$ and $H_2$ is injective;
\item every element of finite order in $Q$ is the image of an element of finite order in $H$ or $G$;  
\item the centre of $Q$ is trivial.
\end{enumerate}
\end{theorem}

\begin{proof} Theorem 1.4 of \cite{AMO} provides a properly relatively hyperbolic group $Q$ satisfying conditions (1)
to (3) and Theorem 2.4(5) of \cite{Osin} controls the elements of finite order in  $Q$, as required in (4).
As $Q$ is properly relatively hyperbolic, its centre $Z(Q)$ is finite. By hypothesis, $H_1$ and $H_2$
have trivial centre, so condition (3) is preserved if we replace $Q$ by $Q/Z(Q)$ and replace $\mu$ by its composition 
with $Q\to Q/Z(Q)$; conditions (1), (2) and (4) are also preserved.
If we continue in this manner, quotienting out the centre of $Q/Z(Q)$ and so on,  the
process will terminate after a finite number of steps because properly relatively hyperbolic groups have a 
maximal finite normal subgroup \cite[Lemma 3.3]{AMO}. Thus  we can arrange for $Q$ to have
 trivial centre.
\end{proof}

\subsection{Infinitely many profinitely-trivial quotients}

The main difficulty in the following proof is proving that the groups $Q_i$ are not abstractly isomorphic, even though one
expects this to be true in great generality.

\begin{theorem}\label{t:lots-of-Q2}  
If $\Delta$ is a non-elementary hyperbolic group, then there is an infinite sequence of non-isomorphic, 
finitely presented groups $Q_i$ and 
epimorphisms $\Delta\to Q_i$ such that  $\widehat{Q}_i=1$. Moreover, each $Q_i$ has trivial center.
\end{theorem}

\begin{proof}  
There are only finitely many conjugacy classes of torsion elements in $\Delta$;
 let $N$ be the maximum of their orders. We fix a
sequence of odd primes
$N<p_1 < p_2<\dots$ and let $D_i$ be the dihedral group of order $2p_i$.
Define $G_i=D_i\ast D_i$.
Let $B$ be a finitely presented infinite group that is torsion-free, centerless, and has no finite quotients; many such groups are known -- see
\cite{BG}, for example. We apply Theorem \ref{t:osin} with $G=G_i$ and $H=B\ast B$; let $R_i$ be the resulting finitely presented common
quotient. As $R_i$ is a quotient of $B\ast B$, it has no non-trivial finite quotients; moreover, by 
item (3) of Theorem \ref{t:osin} it contains a copy of $B$.
As $R_i$ is a quotient of $D_i\ast D_i$ which is not virtually cyclic, it contains torsion elements of order $p_i$. And item (4) of Theorem \ref{t:osin} tells us that the maximum order among the torsion elements of $R_i$ is $p_i$.  

Applying Theorem \ref{t:osin} again with $G=\Delta$ and $H=R_i\ast R_i$, 
we obtain a finitely presented common quotient $Q_i$ that has no non-trivial finite
quotients. Moreover, $Q_i$ is not isomorphic to $Q_j$ if $i\neq j$, because the maximum of the order of a torsion element
in $Q_i$ is $p_i$.
\end{proof}

\subsection{Super-perfect quotients}

\begin{theorem}\label{c:super-perfect}  
Let $\G$ be a finitely generated group that maps onto a 
non-elementary hyperbolic group.
Suppose that $H_2(\G,\Z)=0$.
Then, there is an infinite sequence of
epimorphisms $[\G,\G]\to \tilde{Q}_i $ so that, for all $i\ge 1$,
\begin{enumerate}
\item  $H_1(\tilde Q_i,\Z)=H_2(\tilde Q_i,\Z)=0$;
\item  $\tilde{Q}_i$ is a finitely presented group with no non-trivial finite quotients; and
\item $\tilde{Q}_i/Z_i\not\cong  \tilde{Q}_j/Z_j$ for all $i\neq j$ and all central $Z_i<\tilde{Q}_i$ and $Z_j<\tilde{Q}_j$.
\end{enumerate}
\end{theorem}

\begin{proof} Suppose $\G$ maps onto a non-elementary hyperbolic group $\Delta$. 
Let $Q_i$ be as in Theorem \ref{t:lots-of-Q2} and let $\tilde{Q}_i$ be the universal central extension of $Q_i$. As
$Q_i$ is centerless, it is the quotient of $\tilde{Q}_i$ by its centre,
and indeed the quotient of $\tilde{Q}_i/Z_i$ by its centre for any central $Z_i<\tilde{Q}_i$. 
Therefore $\tilde{Q}_i/Z_i\not\cong  \tilde{Q}_j/Z_j$ if $i\neq j$, as needed for (3).

We compose $\G\to \Delta$ with $\Delta\to Q_i$ to obtain $f:\G\to Q_i$. As $Q_i$ is perfect, this restricts to an epimorphism 
$[\G, \G]\to Q_i$. Lemma \ref{l:schur} assures us this last map lifts to an epimorphism to the universal central extension $[\G, \G]\to \tilde{Q}_i$.
Assertions (1) and (2) follow from Lemma \ref{l:univ}.
\end{proof}

\subsection{Ensuring that fibre products are not abstractly isomorphic}

The awkward hypothesis on centralisers in the following lemma avoids the problem that the centralisers of torsion elements
in hyperbolic groups can be
large. This lemma will be used in much the same manner as it was in \cite[Lemma 5.1]{B-jems}.

\begin{lemma}\label{l:easy} Let  $\Delta$ be a non-elementary hyperbolic group in which the centraliser of every non-trivial 
element is virtually cyclic.  For $i=1,2$, let  $p_i:\Delta\to Q_i$ be an epimorphism with infinite kernel and fibre product $P_i<\D\times\D$. 
If $P_1\cong P_2$ then  $Q_1\cong Q_2$.
\end{lemma}

\begin{proof}  We follow the ideas in the proof of \cite[Lemma 5.1]{B-jems}. Let $N_i = (\Delta\times 1)\cap P_i$ and $M_i = (1\times \Delta)\cap P_i$. We first claim that any isomorphism  
$P_1\cong P_2$ must restrict to an isomorphism $N_1\times M_1\cong N_2\times M_2$.
To prove this, we describe a property that characterises elements of $N_i \cup M_i$ in $P_i$
and is preserved by isomorphisms:
{\em $x \in N_i \cup M_i$ if and only if the centraliser of $x \in P_i$ is not virtually abelian}.  
To see that this characterisation is valid,
note that our assumption on $\Delta$ tells us that the elements of $\Delta\times\Delta$ (hence $P_i$) that do not lie in the direct factors
have virtually abelian centralisers, while the centraliser in $P_i$ of any element of $N_i$ (resp. $M_i$)
contains $M_i$ (resp. $N_i$), which is not virtually abelian since it is an infinite normal subgroup of the non-elementary
hyperbolic group $\Delta$.

Now $N_i\times M_i$ 
is the kernel of the restriction to $P_i$ of $(p_i,p_i):\Delta\times\Delta\to Q_i\times Q_i$, and the image of 
$P_i$ under this map is the diagonal copy of $Q_i$. Thus, $Q_i\cong P_i/(N_i\times M_i)$
is an invariant of the abstract isomorphism  type of $P_i$.
\end{proof}

\begin{lemma}\label{l:P-gives-Q} Let  $\Delta$ be a non-elementary hyperbolic group in which the centraliser of every non-trivial 
element is virtually cyclic.  Let $\phi:G\to\Delta$ be a finitely generated central extension of $\Delta$ and 
for $i=1,2$ let $p_i: G\to Q_i$ be an epimorphism with fibre product $P_i<G\times G$. 
If $P_1\cong P_2$ then  $Q_1/Z_1\cong Q_2/Z_2$, where $Z_i$ is central in $Q_i$. 
\end{lemma}

\begin{proof} We have $\Delta = G/\zeta$ where $\zeta<G$ is central. Let $Z_i=p_i(\zeta)$.
Let $\overline{P}_i$ be the image of $P_i$ in $\Delta\times\Delta$. By hypothesis, the centre of $\Delta$ is trivial,
and since $\overline{P}_i< \Delta\times\Delta$ is a subdirect product, its centre is also trivial. Thus $\overline{P}_i$
is the quotient of $P_i$ by its centre; in particular, $P_1\cong P_2$ implies $\overline{P}_1\cong \overline{P}_2$.

$\overline{P}_i$ is the fibre product of $\Delta\to Q_i/Z_i$, so  $\overline{P}_1\cong \overline{P}_2$
implies $Q_1/Z_1\cong Q_2/Z_2$, by the previous lemma.
\end{proof}

\subsection{Proof of Theorem \ref{t2:lots-of-P}} 
By composing $G\to[\G,\G]$ with the epimorphisms $ [\G,\G]\to \tilde{Q}_i $ furnished by
Corollary \ref{c:super-perfect} we obtain a 
sequence of epimorphisms
$G\to  \tilde{Q}_i $ where each $\tilde{Q}_i$ is a 
super-perfect group with no non-trivial finite quotients and where no central quotient of $\tilde{Q}_i$ is isomorphic to
a central quotient of $\tilde{Q}_j$ if $i\neq j$.  Let $P_i<G\times G$ be the fibre 
product of $G\to \tilde{Q}_i$.  Proposition \ref{p:PT} tells us that $P_i$ is finitely generated
and that $P_i\to G\times G$ induces an isomorphism 
$\widehat{P}_i\overset{\cong}\to\widehat{G\times G}$.  Lemma \ref{l:P-gives-Q} completes the proof.
\qed

\section{Constructing Grothendieck pairs: the last step for Theorem \ref{t:main}} 
\label{s:finish_off}

We would like to apply Theorem \ref{t2:lots-of-P} to the fundamental groups of  arbitrary
Seifert fibre spaces over the base orbifolds  $S^2(p,q,r)$ listed in (\ref{list-top}). More specifically, 
 in the notation of Theorem \ref{t2:lots-of-P}, we would like to take $\Delta=
\Delta(p,q,r)$ and $G=\G = \pi_1M$. 
But we cannot do this because $H_1(M,\Z)$, although finite, is not trivial.
Instead, we look for an auxiliary group $\Lambda$ with finite abelianisation 
and with $H_2(\Lambda,\Z)=0$ so that $\Lambda$ maps onto
a non-elementary hyperbolic group and $\pi_1M$ maps onto a subgroup of finite index in $[\Lambda,\Lambda]$.
In this section  we shall introduce a device that will allow us to achieve this for infinitely many Seifert fibred
spaces over each of the orbifolds in (\ref{list-top}). Our main focus will be on the Seifert fibred
spaces  over
$S^2(3,3,4),\ S^2(3,3,6)$ and $S^2(2,5,5)$. The following theorem completes the proof of Theorem \ref{t:main}.

\begin{theorem}\label{t:333}
If $\Pi$ is the fundamental group of a \SFS whose base orbifold is $S^2(3,3,4),\ S^2(3,3,6)$ or $S^2(2,5,5)$, then
there are infinitely many, pairwise non-isomorphic, finitely generated groups $P\hookrightarrow \Pi\times \Pi$
such that the inclusion induces an isomorphism $\wh{P}\cong \wh{\Pi\times \Pi}$. 
\end{theorem}

We begin with a homological observation.

\begin{lemma} Let $G$ be a finitely generated group and suppose that $N<G$ has index $2$. If $H_2(N,\Z)=0$ and $H_1(N,\Z)$ is a finite group of odd order, then $H_2(G,\Z)=0$.
\end{lemma}

\begin{proof} We consider the LHS spectral sequence in homology for $1\to N\to G\to C_2\to 1$, where $C_2$ is cyclic of order $2$.
The terms on the $E_2$-page that  might contribute to $H_2(G, \Z)$ are 
$H_0(C_2, H_2(N,\Z))$ and $H_1(C_2, H_1(N,\Z))$ and $H_2(C_2, \Z)=0$. The
first is obviously zero and the second is zero because for any coefficent module $M$
one has $H_1(C_2, M) = M/2M$, and we have assumed $|M|$ is odd.
\end{proof}

\subsection{Convenient orbifold quotients of certain \SFSs}

The hypothesis on $|H_1(M,\Z)|$ in the following proposition forces at least two of $p,q,r$ to be odd and it also forces the homology class of the regular fibre (representing the centre of $\pi_1M$) to have odd order in $H_1(M,\Z)$;
this is a serious constraint (see Proposition \ref{p:all-odd}).

\begin{proposition}\label{p:coxeter}
Let $\Pi$ be the fundamental group of a \SFS $M$ with base orbifold $S^2(p,q,r)$. If $|H_1(M,\Z)|$ is finite and odd, 
then there is a group $\Lambda>\Pi$ with the following properties:
\begin{enumerate}
\item $\Pi = [\Lambda, \Lambda]$;
\item $H_1(\Lambda, \Z)$ is cyclic of order $2$;
\item $H_2(\Lambda, \Z) = 0$;
\item $\Lambda$ maps onto $D^-(p,q,r)=\D(p,q,r)\rtimes C_2$.
\end{enumerate}
\end{proposition}

\begin{proof}
$\D=\D(p,q,r)$ is the orientation preserving subgroup of the group $\D^-=\D^-(p,q,r)$ generated by reflections in the sides of the hyperbolic triangle $T=T(p,q,r)$. The short exact sequence $1\to \D \to \D^-\to C_2\to 1$ can be split by lifting the generator of $C_2$ to any one of the three basic reflections; we choose to lift it to the reflection $\tau$ in the edge connecting the vertices fixed by the rotations $a$ and $b$ (in the standard notation). Thus $\D^- = \D\rtimes C_2$ where  the action of $C_2$ on $\D^-$ is $(a,b,c)\mapsto (a^{-1}, b^{-1}, bc^{-1}b^{-1})$.

The reflection $\tau$ descends to a reflection of the orbifold $S^2(p,q,r)$ that interchanges the two connected components of the complement of $\partial T$, and this is covered by
an {\em orientation preserving}
 involution $\tilde{\tau}: M\to M$ that reverses the orientation of fibres. The action of  $\tilde{\tau}$ on $\Pi= \pi_1M$, in the standard notation, is
$\tau_*:(a,b,c,z)\mapsto (a^{-1}, b^{-1}, bc^{-1}b^{-1},z^{-1})$. We define $\Lambda=\Pi\rtimes_{\tau_*}C_2$ to be the resulting semidirect product (i.e. the fundamental group of the orbifold quotient of $M$ by $\tilde{\tau}$).

Conjugation by the generator of $C_2$ sends each generator of $\Pi$ to a conjugate of its inverse, so the image of $\Pi$ in $H_1(\Lambda,\Z)$  is a 2-group, and since $|H_1(\Pi,\Z)|$
is odd, this image must be trivial. This proves (1) and (2), and (3) follows from the preceding lemma, because $H_1(M,\Z)$ is finite, so by 
Poincar\'{e} duality $H_2(M,\Z) = H^1(M,\Z)=0$, and $H_2(\Lambda,\Z)$ is a quotient of $H_2(M,\Z)$
(with equality if $\frac{1}{p} + \frac{1}{q}+\frac{1}{r} < 1$, because then $M$ is aspherical). 
\end{proof}

\subsection{Proof of Theorem \ref{t:333}}
 As was explained in Section 4, an arbitrary \SFS over $S^2(3,3,4)$ has fundamental
group
$$
\Pi=\<a,b,c,z\mid z \text{ is central },\ a^3z^{e_1}=b^{3}z^{e_2}=c^4z^{e_3}=1, abc = z^d\>
$$
where $e_1, e_2\in\{1,2\}$ and $e_3\in\{1,3\}$. We calculated in Lemma \ref{small_ab} that
the abelianisation of $\Pi$ has order $3 (4 e_1 + 4 e_2 + 3 e_3 + 12 d)$, which is odd in all 
cases.  Likewise, the calculations in Lemma \ref{small_ab} show that for every
\SFS $M$ with base $S^2(3,3,6)$ or $S^2(2,5,5)$, the order of $H_1(M,\Z)$ is odd:
for $(3,3,6)$ we have $9(2 e_1 + 2 e_2 + e_3 + 6 d)$ with $e_3$ odd, 
and for $(2,5,5)$ we have  $5(5 + 2 e_2 + 2 e_3 + 10 d)$.
Thus, in all cases, $\Pi$ satisfies the hypotheses of Proposition \ref{p:coxeter},
which furnishes us with a group $\Lambda$ so that  Theorem \ref{t2:lots-of-P} 
applies with the appropriate $\Delta(p,q,r)$ in the role of $\Delta$, $\Pi$ in the role of $G$ and $\Lambda$ in the role of $\Gamma$.
\qed

\section{Infinitely many examples over each base orbifold in list (\ref{list-top})}

When combined with Theorems \ref{t3},  \ref{t:333} and \ref{t:not-fp}, the results in this section show that for infinitely many Seifert fibred spaces 
 $M$ over each of the bases listed in (\ref{list-top}), the fundamental group $\G=\pi_1M$
satisfies Theorem \ref{t:main} parts (1), (2) and (3).

\subsection{Infinitely many examples over $S^2(3,3,5)$}

\begin{proposition}\label{p:all-odd} If $p, q$ and $r$ are odd,
then for every $d\in\Z$ the \SFS over $S^2(p,q,r)$ with fundamental group  
$$
\Pi=\<a,b,c,z\mid a^p=b^q= c^r=z,\ abc=z^{2d} \>
$$
is such that $\Pi\times\Pi$ has infinitely many non-isomorphic, finitely generated subgroups $P$
such that  $P\hookrightarrow \Pi\times\Pi$ is a Grothendieck pair.
\end{proposition}

\begin{proof} 
As in the proof of Lemma \ref{small_ab}, the determinant of the relation matrix can be computed and equals (in absolute value) $r(q+p) + pq - 2pqrd$, which is odd.
Thus $\Pi$ satisfies the hypotheses of Proposition \ref{p:coxeter}. That proposition furnishes us with a group $\Lambda$ so that  Theorem \ref{t2:lots-of-P} 
applies with the appropriate $\Delta(p,q,r)$ in the role of $\Delta$, $\Pi$ in the role of $G$, and $\Lambda$ in the role of $\Gamma$.
 \end{proof}

 \begin{corollary}\label{c:335}
 For infinitely many \SFSs $M$ over $S^2(3,3,5)$, the fundamental group $\G=\pi_1M$
 satisfies the requirements of Theorem \ref{t:main}.
 \end{corollary}

 \begin{proof}  
 Theorem \ref{t3} tells us that any finitely generated, residually finite group $P$  with the same
 profinite completion as $\G\times\G$ arises from a Grothendieck pair, and Theorem \ref{t:not-fp}
 tells us that $P$ cannot be finitely presented if it is not isomorphic to $\G\times\G$.
  Proposition \ref{p:all-odd} provides infinitely many possibilities for $P$.
 \end{proof}

 \begin{remark} A slight variation on the proof of Proposition  \ref{p:all-odd}  shows that 
if $p$ and $q$ are odd and $r$ is even, then for every $d\in\Z$ the \SFS over $S^2(p,q,r)$ with fundamental group  
$$
\Pi=\<a,b,c,z\mid a^p=b^q= c^r=z,\ abc=z^{d} \>
$$
is such that $\Pi\times\Pi$ has infinitely many non-isomorphic, finitely generated subgroups $P$
such that  $P\hookrightarrow \Pi\times\Pi$ is a Grothendieck pair. The reader will readily  devise 
other conditions on the Seifert data that ensure the same conclusion, but
some care is needed. For instance, it is important in Proposition \ref{p:all-odd} that the exponent of $z$ in the relations $a^p=b^q= c^r=z$ is $1$.  By way of contrast, with $p=q=r=5$,  the \SFS $M$ over $S^2(5,5,5)$ with Seifert invariants 
 $(5,2), (5,2), (5,1)$ and $d=1$ has $H_1(M,\Z)$ infinite and $e(M)=0$.
 In this case $\pi_1M$ is not profinitely rigid, by \cite{Hem}. \end{remark}

\subsection{Infinitely many examples over $S^2(4,4,4)$}  

We want to exhibit  infinitely many \SFS $M$ over $S^2(4,4,4)$ so that the fundamental group $\G=\pi_1M$
 satisfies the requirements of Theorem \ref{t:not-GR} and hence Theorem \ref{t:main} parts (1), (2) and (3).
In the light of Lemma \ref{l:usual} and Theorem \ref{t:333}, it suffices to prove that infinitely many such $\G$
arise as a subgroup of finite index in the fundamental group $\Pi$ of a Seifert fibred space with base $S^2(3,3,4)$.
To do this, we simply pull back the commutator subgroup $\D(4,4,4)<\D(3,3,4)$, which has index $3$.
$$ 
\begin{matrix}
&1\longrightarrow &\Z&\longrightarrow&
\G &\longrightarrow& \Delta(4,4,4)&  \longrightarrow& 1
\cr
&&{=} {\bigg\downarrow} &&{\bigg\downarrow} 
&& {\bigg\downarrow}  &&
\cr
&1\longrightarrow &{\Z}&\longrightarrow&
\Pi  &\longrightarrow& \Delta(3,3,4) &  \longrightarrow& 1
\end{matrix}
$$  
The Euler number $e(\G)=3 e(\Pi)$ varies without bound as we vary $\Pi$, so we obtain infinitely many
possibilities for $\G$. Indeed one can calculate that $\G$ can be any of the groups  
$$
\G_+(d) = \< a,b,c,z\mid a^4z,\ b^4z,\ c^4z,\, abcz^d \>.
$$
The Seifert fibred spaces  
that are not covered by this argument are those whose fundamental group lies in the family
$$
\G_-(d) = \< a,b,c,z\mid a^4z^{-1},\ b^4z,\ c^4z,\, abcz^d \>.
$$
 
\section{Closing Remarks}\label{s:last}
It seems reasonable to expect that the fundamental group $\G$ of every \SFS over each of the base orbifolds
from lists (\ref{list-top}) and (\ref{list-bottom}) satisfies the conclusion of Theorem \ref{t:main}. At this point, we are unable to prove that {\em any} \SFS over a base orbifold from (\ref{list-bottom}) satisfies the conclusion of Theorem \ref{t:main}, but with our other results in hand, we would be able to do this if we could map $\G$ onto infinitely many, finitely presented,
non-isomorphic groups $Q$ with $\wh{Q}=1$ and $H_2(Q,\Z)=0$. The nub of our remaining difficulties is
that we do not see how to arrange the condition  $H_2(Q,\Z)=0$ in sufficient generality. 
We have exploited Theorem \ref{t2:lots-of-P} to overcome this difficulty in many cases  
but
 \SFSs over bases in (\ref{list-bottom}) are not covered by the arguments that we have presented.
If one could construct suitable quotients $Q$ for one \SFS over any base orbifold, then one could
construct infinitely many such \SFSs over the same base, as we shall now explain.

\subsection{Finite extensions and circle actions}

We begin with a general tool for promoting the existence of Grothendieck
pairs in a group to the existence of pairs in certain finite extensions of the group.

\begin{lemma}\label{l:jack-up}
Let $1\to G_0\to G\to Q\to 1$ be a short exact sequence of groups with $G$ finitely generated 
and residually finite, and assume that
$Q$ is a finite group generated by the image of a finitely generated normal subgroup $S<G$. 
Suppose that  
$P_0\hookrightarrow G_0$ is a Grothendieck pair with $P_0$
finitely generated and that $S\cap G_0 \subset P_0$.
Let $P= \< P_0, S \>$. 
Then $P\hookrightarrow G$ is a Grothendieck pair.
\end{lemma}

\begin{proof} As $S\cap G_0 \subset P_0$, we have $P\cap G_0 = P_0$ and therefore
$P_0$ has finite index in $P$. In more detail,
each $g\in P$ can be written $g=s g_0$ with $g_0\in P_0<G_0$ and $s\in S$, so if $g\in P\cap G_0$
then $s \in S\cap G_0 \subset P_0$, hence $g\in P_0$. We are assuming   
$S$ maps onto $Q$, so we have the following commutative diagram, where the vertical maps are inclusions
$$ 
\begin{matrix}
&1\longrightarrow &P_0&\longrightarrow&
P &\longrightarrow& Q&  \longrightarrow& 1
\cr
&&\ {\bigg\downarrow} &&{\bigg\downarrow} 
&& {\bigg\downarrow} {=}&&
\cr
&1\longrightarrow &G_0&\longrightarrow&
G &\longrightarrow& Q &  \longrightarrow& 1.
\end{matrix}
$$ 
As $P_0$ has finite index in the finitely generated group $P$, the subspace topology on the closure of $P_0$ in $\wh{P}$
is the full profinite topology, and similarly for $G_0<G$. Thus we have a commutative diagram of profinite groups
$$ 
\begin{matrix}
&1\longrightarrow &\wh{P}_0&\longrightarrow&
\wh{P} &\longrightarrow& Q&  \longrightarrow& 1
\cr
&&\ {\bigg\downarrow} &&{\bigg\downarrow} 
&& {\bigg\downarrow} {=}&&
\cr
&1\longrightarrow &\wh{G}_0&\longrightarrow&
\wh{G} &\longrightarrow& Q &  \longrightarrow& 1.
\end{matrix}
$$ 
The first vertical map is an isomorphism, by hypothesis, and hence the second vertical map is as well. 
\end{proof}

We apply this lemma to the map on fundamental groups induced by  
a regular covering arising from the natural circle action on a Seifert fibred space.
Recall (from \cite{neumann}, for example) 
that associated to any Seifert fibred space $M$ one has a continuous action of ${S}^1$ that preserves the circle fibres; regular fibres of the fibration are free orbits, while a singular fibre over a cone point of order $p$ will have a stabilizer that is cyclic of order $p$. If $N$ is an integer that is coprime to the orders of all of the cone points, then the cyclic group $C_N<S^1$ consisting of $N$-th roots of unity will act freely on $M$ and we can form the quotient manifold $M_N = M/C_N$. At the level of $\pi_1$, passing from $\pi_1M$ to $\pi_1M_N$ corresponds to augmenting the centre (which is the homotopy class of a regular fibre) by adjoining a central $N$-th root. Thus if
$$
\pi_1M = \< a,b,c, z\mid z\text{ is central}, a^p=z^{\beta_1},\ b^q=z^{\beta_2},\ c^r=z^{\beta_3},\ abc=z^d\>,
$$
then
$$
\pi_1M_N = \< a,b,c, \zeta\mid \zeta\text{ is central}, a^p=\zeta^{N\beta_1},\ b^q=\zeta^{N\beta_2},\ c^r=\zeta^{N\beta_3},\ abc=\zeta^{Nd}\>,
$$
where $\zeta^N=z$.

The Euler numbers $e(M_N) = N\, e(M)$ assure us that the manifolds $M_N$ are pairwise distinct if $e(M)\neq 0$.

The labelling of generators defines a short exact sequence
\begin{equation} 
1\to \pi_1 M \to \pi_1 M_N \to C_N \to 1,
\end{equation} 
where $C_N$ is the image of $\zeta\in \pi_1M$.
The key point to observe is that  this sequence and its square
 provide us with  settings in which Lemma \ref{l:jack-up} can be applied. 
 Thus we obtain:

\begin{proposition}\label{p:last} Let $M$ be a Seifert fibred space with base oribifold $\mathcal{O}$ and Euler
number $e(M)\neq 0$. If there are finitely many (respectively, infinitely many) Grothendieck
pairs $P\hookrightarrow \pi_1M\times\pi_1M$ with $P$ finitely generated,
then there are finitely many (respectively, infinitely many) Grothendieck
pairs $P_N\hookrightarrow \pi_1M_N\times\pi_1M_N$ for infinitely many \SFSs $M_N$ over $\mathcal{O}$.
\end{proposition}

\begin{example} Let $M$ be the Seifert fibred space over $S^2(4,4,4)$ with fundamental
group  
$$
\G_+(1) = \< a,b,c, z\mid z\text{ is central }, a^4= b^4= c^4=z=abc \>.
$$
Following the construction described above, we add $N$-th roots to the centre with $N$ a positive odd integer.  
If $N=4k+1$, then
setting $\zeta^N=z$ and writing $A=a\zeta^{-k},\ B=b\zeta^{-k},\ C=c\zeta^{-k}$ we get
$$
\G_+(k+1) = \< A,B,C, \zeta \mid  A^4= B^4= C^4=\zeta,\ ABC=\zeta^{k+1}\>.
$$
For $N=4k-1$, setting $\zeta^N=z^{-1}$ and  writing $A=a\zeta^{k},\ B=b\zeta^{k},\ C=c\zeta^{k}$ we get
$$
\G_+( -k+1 )= \< A,B,C, \zeta\mid A^4= B^4= C^4=\zeta,\ ABC=\zeta^{-k+1} \>.
$$ 
$\G_+(1) $ is the commutator subgroup of a \SFS over $S^2(3,3,4)$, so we obtain
infinitely many Grothendieck pairs in $\G_+(1)\times\G_+(1)$ by applying Theorem \ref{t2:lots-of-P}.
Proposition \ref{p:last} tells us that the same is true of all of the groups $\G_+(d)$, 
and this gives a second proof  of Theorem \ref{t:main} for these groups.
\end{example}

\subsection{A more direct argument for \SFSs over $S^2(4,4,4)$ and $S^2(3,3,4)$}

In our proof of Theorem \ref{t:main}, we appealed to the work of  Wilkes on
the relative profinite rigidity of Seifert fibred spaces \cite{Wil}: we invoked it to
prove Lemma \ref{l:dense-will-do}. This appeal to \cite{Wil}
can be avoided in the case of \SFSs over $S^2(4,4,4)$ and $S^2(3,3,4)$,
as we shall now explain.  
The fundamental group of each \SFS over $S^2(4,4,4)$ is isomorphic to one
of the following groups
$$
\G_+(d) = \< a,b,c,z\mid a^4z,\ b^4z,\ c^4z,\, abcz^d \>,
$$
$$
\G_-(d) = \< a,b,c,z\mid a^4z^{-1},\ b^4z,\ c^4z,\, abcz^d \>.
$$
For a prime $p$, 
the {\em exponent-$p$ lower central series}
 of a group $\G$ is defined recursively by $\G_0=\G$ and $\G_{i+1}=\G_i^{(p)}\,
 [\G_i, \G]$,
where the superscript denotes the group generated by $p$-th powers. Let $N_i^\G = \G/\G_i$.
For us, $p=2$; for emphasis, we may write $\G_i(2)$ and $N_i^\G(2)$  instead of $\G_i$ and $N_i^\G$.

\begin{theorem}\label{t:tell-444}  A \SFS
with base orbifold $S^2(4,4,4)$ and fundamental group $\G$  is uniquely determined by $H_1(\G,\Z)$ and
$N_5^\G(2)$, and in all cases  $|N_5^\G(2)|=2^{17}$.
\end{theorem} 

\noindent{\em Sketch of proof.}  
The abelianisation of $\G_+(d)$ is $C_4 \times C_4 \times C_{|4d-3|}$ 
while the abelianisation of $\G_-(d) $ is $C_4 \times C_4 \times C_{|4d-1|}$, so abelianisation
alone reduces our task to distinguishing between $\G_+(d)$ and $\G_-(1-d)$. It is easy to verify
that for all integers $d$, the quotient  $\G_+(d)/\<z^8\>$ is isomorphic to either $\G_+(0)/\<z^8\>$ or $\G_+(1)/\<z^8\>$,
while $\G_-(d)/\<z^8\>$ is isomorphic to $\G_-(0)/\<z^8\>$ or $\G_-(1)/\<z^8\>$. So if we can argue that 
for every $\G=\G_{\pm}(d)$, the
order of the image  of $z$  in $N_5^\G(2)=\G /\G_5(2)$ divides $8$,
 then it will be enough to verify (i) that $|N_5^\G(2)|=2^{17}$
for $\G_{\pm}(0)$ and $\G_{\pm}(1)$, and (ii) that $N_5^{G_+}(2)$ is not isomorphic to $N_5^{G_-}(2)$
if $G_+\in \{\G_+(0),\, \G_+(1)\}$ and $G_-\in \{\G_-(0),\, \G_-(1)\}$; this verification can be made
using the   {\rm{pQuotient}}
package in  {\rm{Magma}} \cite{Mag}.
To see that the order of the image $z$ does
divide $8$, note that since $z$ has odd order in $H_1(\G,\Z)$, its image in  $N_1^\G =C_2\times C_2$ is trivial,
and since $z^{\pm 1} = a^4 \in \G_1^2<\G_2$, the image of $z$ in $\G/\G_2$ is also trivial.
It follows that the image of $z$ in $N_3^\G$ has order at most $2$, 
while  in $N_4^\G$ its order divides $4$, and in $N_5^\G$ its order divides $8$. 
\qed

This provides the desired substitute for Lemma \ref{l:dense-will-do}.

\begin{corollary} \label{c:bingo}
If $M\neq M'$ are \SFSs with base orbifold $S^2(4,4,4)$ and $H_1(M,\Z)\cong H_1(M',\Z)$, then $\pi_1M'$
does not embed densely in $\wh{\pi_1M}$.  
\end{corollary}

\begin{proof} Let $\G=\pi_1M$ and $\G'=\pi_1M'$.
If $\G'$ were to embed densely in $\G$,   it would map onto $N^\G_5(2)$. But such a map
$\G'\to N^\G_5(2)$ has to factor through $N^{\G'}_5(2)$.
Since $|N^\G_5(2))|= |N^{\G'}_5(2)|$, these groups must be equal, which forces $\G=\G'$ and $M=M'$.
\end{proof}

One can deduce the same result for \SFSs with base orbifold $S^2(3,3,4)$ by exploiting the 
fact that $\D(4,4,4)$ has index $3$ in $\D(3,3,4)$. We omit the details.

  
\end{document}